\documentclass[11pt]{article}
\usepackage[utf8]{inputenc}
\usepackage[T1]{fontenc}
\usepackage{amsmath}
\usepackage{amsfonts}
\usepackage{amssymb}
\usepackage{graphicx}
\usepackage[left=1.00in, right=1.00in, top=1.00in, bottom=1.00in]{geometry}
\usepackage{amsthm}

\usepackage[authoryear]{natbib} %% uncomment this for author-year bibliography
\usepackage{breakcites}
\RequirePackage[colorlinks,citecolor=blue,urlcolor=blue]{hyperref}
\RequirePackage{graphicx}
\usepackage[ruled,vlined]{algorithm2e}

\usepackage{subfig}
\usepackage{multicol}
\usepackage{tikz, pgfplots}
\usepackage{authblk}
\usepackage{tikz-3dplot}
\usetikzlibrary{shapes,calc,positioning}
\tdplotsetmaincoords{70}{120}
\pgfplotsset{compat=1.16}
\usetikzlibrary[arrows.meta,bending]
\usetikzlibrary{positioning}
\usetikzlibrary{patterns}

\theoremstyle{plain}

\newtheorem{theorem}{Theorem}[section]
\newtheorem{lemma}[theorem]{Lemma}
\newtheorem{corollary}[theorem]{Corollary}
\newtheorem{proposition}[theorem]{Proposition}

\theoremstyle{remark}
\newtheorem{remark}[theorem]{Remark}
\newtheorem{definition}[theorem]{Definition}
\newtheorem{example}{Example}

\usepackage{booktabs}
\usepackage{tikz, pgfplots}

\pgfplotsset{compat=1.16} 
\newcommand{\pr}{\normalfont\textrm{Pr}}
\newcommand{\diff}{\, \mathrm{d}}
\newcommand{\topic}[1]{#1}

\title{Exchangeable FGM copulas}
\author{Christopher Blier-Wong\thanks{Corresponding author: Christopher Blier-Wong, christopher.blier-wong.1@ulaval.ca. Address: 2425, rue de l'Agriculture, office 00154A. Québec (Québec) G1V 0A6 Canada}}
\author{Hélène Cossette}
\author{Etienne Marceau}% \thanks{Corresponding author: Etienne Marceau, etienne.marceau@act.ulaval.ca. Address: 2425, rue de l'Agriculture, office 4177. Québec (Québec) G1V 0A6 Canada}
\affil{École d'actuariat, Université Laval, Québec, Canada}
\date{January 15, 2022} % Soumission JAP
\begin{document}

\maketitle

\begin{abstract}
Copulas are a powerful tool to model dependence between the components of a random vector. One well-known class of copulas when working in two dimensions is the Farlie-Gumbel-Morgenstern (FGM) copula since their simple analytic shape enables closed-form solutions to many problems in applied probability. However, the classical definition of high-dimensional FGM copula does not enable a straightforward understanding of the effect of the copula parameters on the dependence, nor a geometric understanding of their admissible range. We circumvent this issue by studying the FGM copula from a probabilistic approach based on multivariate Bernoulli distributions. This paper studies high-dimensional exchangeable FGM copulas, a subclass of FGM copulas. We show that dependence parameters of exchangeable FGM can be expressed as convex hulls of a finite number of extreme points and establish partial orders for different exchangeable FGM copulas (including maximal and minimal dependence). We also leverage the probabilistic interpretation to develop efficient sampling and estimating procedures and provide a simulation study. Throughout, we discover geometric interpretations of the copula parameters that assist one in decoding the dependence of high-dimensional exchangeable FGM copulas. 
\end{abstract}


\section{Introduction}

\topic{Copulas are a powerful tool to model dependence between the components of a random vector.} A well-known family of copulas is the Farlie-Gumbel-Morgenstern (FGM) copulas, first studied by \cite{eyraud1936principes}, \cite{morgenstern1956einfache}, \cite{farlie1960performance}, and \cite{gumbel1960bivariate}. 

\topic{FGM copulas are attractive since their simple shape enables exact calculus.} Being quadratic in each marginal, FGM copulas allow one to develop closed-form expressions for many quantities of interest. The properties of FGM copulas, from a mathematical perspective, are well established, see, for instance, \cite{cambanis1977some}, \cite{johnson1975generalized}, \cite{mai2014financial}, \cite[chapter 5]{kotz2001correlation}, \cite[Section 44.10]{kotz2004continuous}, \cite{durante2015principles} or \cite{nelsen2007introduction}. Given a set of dependence parameters, FGM copulas have been applied in many applications, including in, for instance, finance (\cite{mai2014financial}), actuarial science (\cite{barges2011moments}), bioinformatics (\cite{kim2008copula}) and hydrology \cite{genest2007everything}. However, little is known about the interpretation of the FGM copula parameters for higher dimensions, in particular, how the copula parameters affect the dependence structure and how one may compare dependence constructions in terms of dependence orders (see Section 3.9 in \cite{muller2002comparison} for details on the latter topic).

\topic{In \cite{blier-wong2022stochasticb}, the authors establish a one-to-one correspondence between FGM copulas and symmetric multivariate Bernoulli distributions.} One advantage of this representation is that one may construct subfamilies of FGM copulas by selecting subfamilies of multivariate symmetric Bernoulli distributions. Then, the subfamily of FGM copulas will share the dependence properties of the symmetric Bernoulli distributions, thus simplifying the parameter space and related operations like sampling and estimation. Another advantage of the stochastic representation of FGM copulas is that multivariate Bernoulli distributions are simpler to understand. While FGM copulas only induce weak dependence, they are the simplest of the Bernstein copulas. A $d$-dimensional Bernstein copula, introduced by \cite{sancetta2004bernstein}, is dense on the hypercube $[0, 1]^d$, but many dependence parameters are required to specify the copula. Indeed, due to their flexibility, Bernstein copulas are sometimes used as alternatives to the empirical copula \cite{segers2017empirical}. Understanding the stochastic nature of FGM copulas is essential preliminary work before investigating the properties, geometries and stochastic nature of Bernstein copulas.

% \topic{One class of copulas that has not received much attention are the Bernstein copulas}. A $d$-dimensional Bernstein copula, introduced by \cite{sancetta2004bernstein}, is dense on the hypercube $[0, 1]^d$, but many dependence parameters are required to specify the copula. Indeed, due to their flexibility, Bernstein copulas are sometimes used as alternatives to the empirical copula \cite{segers2017empirical}. However, their practical applications in risk modelling remain sparse. 

% \topic{Applications of high-dimensional copulas, as discussed in \cite{grosser2021copulae}, have recourse to (i) hierarchical Archimedian copulas, (ii) Archimax copulas, (iii) factor copulas or (iv) vine copulas.} The main advantage of these models is that they represent a good balance between flexibility and parsimony. Indeed, while Archimedian copulas are popular in two dimensions, they are not practical for higher dimensions since they typically have only one dependence parameter that specifies the Archimedian copula generator. On the other hand, Bernstein copulas have too many parameters for practical uses. Indeed, the simplest Bernstein copula is the FGM copula, and $d$-dimensional FGM copulas require $2^d - d - 1$ dependence parameters that satisfy a set of $2^d$ inequalities. Since one requires multivariate copulas that are both flexible and parsimonious, one may achieve this goal by constructing a FGM copula by selecting a subfamily of symmetric multivariate Bernoulli distribution that depends on few parameters, but which has the desired dependence structure. 

\topic{The present paper introduces exchangeable FGM (eFGM) copulas}. The eFGM copulas are subfamilies of FGM copulas that we construct with exchangeable symmetric multivariate Bernoulli random vectors. It follows that exchangeable FGM copulas are the simplest of the exchangeable Bernstein copulas and that understanding the geometry and properties of eFGM copulas lays the groundwork to understand the properties of exchangeable Bernstein copulas. In turn, one will be in a better position to understand the general class of Bernstein copulas. The exchangeability assumption is reasonable and useful in some contexts, consider, for instance, the study of litter-mates in laboratory experiments (\cite{kuk2004litterbased}), finance (\cite{perreault2019detection}), reliability theory (\cite{navarro2005note}), actuarial science (\cite{kolev2008random}) or credit default risk (\cite{fontana2021model}). Exchangeability also plays an important role in Bayesian statistics (\cite{schervish2012theory}). 

\topic{Another advantage of studying the class of eFGM copulas is that the FGM copula corresponding to the lower bound under the supermodular order is a special case of the class of eFGM copulas}. We study this lower bound in detail in Section \ref{sec:dependence-ordering}. We also introduce subfamilies of eFGM copulas that display a specific shape of dependence structure, and one may compare copulas under the supermodular order within the subfamilies. The ability to order random vectors with respect to the supermodular order is important for practical applications. For instance, in applied probability, finance, and actuarial science, one may be interested in the aggregate distribution of a random vector. If one may order two copulas under the supermodular order, then one may order the two aggregate distributions under the stop-loss order, which implies inequalities with respect to certain useful risk measures. See, for instance, Section 8.3 of \cite{muller2002comparison} or Section 6.3 of \cite{denuit2006actuarial} for details on the supermodular order, the stop-loss order and aggregate distributions. 

\topic{The remainder of this paper is organized as follows.} In Section \ref{sec:definition}, we introduce the subclass of eFGM copulas. Section \ref{sec:construction} presents construction methods for symmetric exchangeable Bernoulli random variables, and their relationship to eFGM copulas. In Section \ref{ss:rays}, we show that the parameters of all eFGM copulas can be expressed as a convex hull of eFGM copula dependence parameters. We also provide a method for analytically obtaining the set of extreme points corresponding to the copula parameters. It follows that one can view the subfamily of eFGM copulas as a finite mixture model. By studying the extreme points of the convex hull of the dependence parameters of eFGM copulas, we gain a geometric understanding of the class of eFGM copulas. We deal with dependence ordering in Section \ref{sec:dependence-ordering}, providing methods to compare eFGM copulas under the supermodular order. In Section \ref{sec:sample-estimation}, we discuss sampling and estimation for high dimensional eFGM copulas. In particular, we leverage the stochastic representation of eFGM copulas to propose an efficient stochastic sampling method, and we leverage the finite mixture representation to propose an estimation algorithm. Certain proofs are deferred to the Apprendix.

\section{Definition}\label{sec:definition}

In this section, we introduce the subfamily of copulas studied in the paper. Copulas are multivariate cumulative distribution functions (cdf) of rvs with uniform marginals. 
\begin{definition}\label{def:copula}
	A ($d$-variate) copula is a function $C: [0, 1]^d \to [0, 1]$ satisfying
	\begin{enumerate}
		\item $C(u_1, \dots, u_d) = 0$ if any $u_j = 0$, $j \in \{1, \dots, d\}$.
		\item $C(u_1, \dots, u_d) = u_j$ if $u_k = 1$  for all $k \in \{1, \dots, d\}$ and $k \neq j$.
		\item $C$ is $d$-increasing on $[0, 1]^d$, that is,
		$$\sum_{i_1 = 1}^2\dots \sum_{i_d = 1}^2 (-1)^{i_1 + \dots + i_d}C(u_{1i_1}, \dots, u_{di_d}) \geq 0$$
		for all $0 \leq u_{j1} \leq u_{j2} \leq 1$ and $j \in \{1, \dots, d\}.$
	\end{enumerate}
\end{definition}

An important family of copulas is the FGM copulas, first studied by \cite{eyraud1936principes}, \cite{farlie1960performance}, \cite{gumbel1960bivariate} and \cite{morgenstern1956einfache}. One may refer to \cite{cambanis1977some}, \cite{johnson1975generalized}, Chapter 5 of \cite{kotz2001correlation}, Section 44.10 of \cite{kotz2004continuous} or \cite{genest2007everything} for properties of this family of copulas. An FGM copula is defined as
\begin{equation} \label{fgmco1205}
	C\left( u_1, \dots, u_d\right) =\left(\prod_{j=1}^d u_j\right) \left( 1+\sum_{k=2}^{d}\sum_{1\leq j_{1}<\dots <j_{k}\leq d}\theta_{j_{1}\dots j_{k}}\overline{u}_{j_{1}}\overline{u}_{j_{2}}\dots \overline{u}_{j_{k}}\right)
	\quad (u_1, \dots, u_d)\in [0,1]^{d},
\end{equation}
where $\overline{u}_{j}=1-{u}_{j}$, $j \in \{1,\dots,d\}$. The constraints on the parameters of FGM copulas, derived in \cite{cambanis1977some}, are as follows:
\begin{equation}\label{eq:constraints-general}
	\left\{(\theta_{12}, \dots, \theta_{1\dots d}) \in \mathbb{R}^{2^d - d - 1} : 1+\sum_{k=2}^{d}\sum_{1\leq j_{1}<\cdots <j_{k}\leq d}\theta_{j_{1}\ldots j_{k}}\varepsilon _{j_{1}}\varepsilon _{j_{2}}\ldots
	\varepsilon_{j_{k}}\geq 0\right\},  
\end{equation}
for $\{\varepsilon _{j_{1}},\varepsilon _{j_{2}},\ldots, \varepsilon_{j_{k}}\} \in \{-1,1\}^d$. We call the $\binom{d}{k}$ parameters $\theta_{j_{1}\dots j_{k}}$, for $1 \leq j_1 < \dots < j_k \leq d$ the $k$-dependence parameters, $k \in \{2, \dots, d\}$. A $d$-variate FGM copula has $2^d - d - 1$ parameters, the large number of parameters becomes impractical for large dimensional applications of FGM copulas. The current paper studies the subfamily of exchangeable FGM copulas that have the shape
\begin{equation}\label{eq:eFGM-copula}
	C_d(u_1, \dots, u_d) = \left(\prod_{j = 1}^d u_j\right)\left(1 + \sum_{k = 2}^d \sum_{1\leq j_1< \dots < j_k\leq d} \theta_k \overline{u}_{j_1} \dots \overline{u}_{j_k}\right), \quad (u_1, \dots, u_d) \in [0, 1]^d.
\end{equation}
For $k \in \{2, \dots, d\}$, this class of FGM copulas sets each of the $\binom{d}{k}$ parameters $\theta_{j_{1}\dots j_{k}} = \theta_k$ for all $1 \leq j_1 < \dots < j_k \leq d$, that is, every $k$-dependence parameters are equal. By symmetry of the bivariate FGM copula, it is obvious that with $d = 2$, each admissible parameter $\theta_2 \in [-1, 1]$ corresponds to an exchangeable bivariate FGM copula parameter, that is, the entire class of bivariate FGM copulas are also eFGM copulas. 

A $d$-variate eFGM copula is specified by a vector of $d - 1$ parameters $(\theta_2, \dots, \theta_d) \in \mathcal{T}_d$ (as opposed to $2^d - d - 1$ for the complete class of $d$-variate FGM copulas) where, for $\{\varepsilon_{1},\dots, \varepsilon_{d}\} \in \{-1,1\}^d$, the following constraints are satisfied:
\begin{equation}\label{eq:constraints-exch}
	\mathcal{T}_d=\left\{(\theta_2, \dots, \theta_d) \in \mathbb{R}^{d-1} : 1+\sum_{k=2}^{d}\sum_{1\leq j_{1}<\cdots <j_{k}\leq d}\theta_{k}\varepsilon _{j_{1}}\dots
	\varepsilon _{j_{k}}\geq 0\right\},
\end{equation}
As the dimension $d$ increases, satisfying the $2^{d-1}$ constraints for the parameters $(\theta_2, \dots, \theta_d)$ in \eqref{eq:constraints-exch} becomes tedious. A preferable approach will rely on the following stochastic representation introduced in \cite{blier-wong2022stochasticb}.

\begin{theorem}\label{thm:stochastic-representation}
	The copula in \eqref{eq:eFGM-copula} has an equivalent representation 
	\begin{equation}\label{eq:stochastic-formulation}
		C(u_1, \dots, u_d) = \sum_{\{i_1, \dots, i_d\} \in \{0, 1\}^d} f_{I_1, \dots, I_d}(i_1, \dots, i_d) \prod_{m = 1}^d u_m \left(1 + (-1)^{i_m}\overline{u}_m\right),
	\end{equation}
	for $\left(u_1, \dots, u_d\right) \in [0, 1]^d$, where $f_{I_1, \dots, I_d}$ is the probability mass function (pmf) of $\left(I_1, \dots, I_d\right)$, an exchangeable, symmetric multivariate Bernoulli random vector with 
	\begin{equation}\label{eq:theta_from_I}
		\theta_k = (-2)^k E\left\{\prod_{j = 1}^k\left(I_{j} - \frac{1}{2}\right)\right\}, \quad k \in \{2, \dots, d\}.
	\end{equation}
\end{theorem}
\begin{proof}
	The relationship in \eqref{eq:stochastic-formulation} applies to every FGM copula by Corollary 3.3 of \cite{blier-wong2022stochasticb}. For eFGM copulas, one must show that $\left(I_1, \dots, I_d\right)$ is an exchangeable Bernoulli random vector. By Theorem 3.2 of \cite{blier-wong2022stochasticb}, the relationship between the parameters $(\theta_2, \dots, \theta_d)$ and the underlying multivariate symmetric Bernoulli random vector for FGM copulas is $\theta_{j_1 \dots j_k} = (-2)^k E\left\{\prod_{l = 1}^k\left(I_{j_l} - \frac{1}{2}\right)\right\},$ for $1 \leq j_1 < \dots < j_k \leq d$ and $k \in \{2, \dots, d\}$. Hence, for $\theta_{j_1 \dots j_k} = \theta_{j_1' \dots j_k'}$ for all $1\leq j_{1}<\cdots <j_{k}\leq d$ and $1\leq j_{1}'<\cdots <j_{k}'\leq d$, we must have $f_{I_{j_{1}},\ldots,I_{j_{k}}}\left(i_{j_{1}},\ldots,i_{j_{k}}\right) = f_{I_{j_{1}'},\ldots,I_{j_{k}'}}\left(i_{j_{1}'},\ldots,i_{j_{k}'}\right)$ for all $1\leq j_{1}<\cdots <j_{k}\leq d$ and $1\leq j_{1}'<\cdots <j_{k}'\leq d$, with $k = \{2, \dots, d\}$, which the pmf of exchangeable Bernoulli random vectors satisfy.
\end{proof}

\section{Construction methods and examples}\label{sec:construction}

We first present a few methods to construct exchangeable symmetric Bernoulli rvs. Alternating between different construction methods, along with the stochastic representation of eFGM copulas in Theorem \ref{thm:stochastic-representation}, will enable us to study properties of eFGM copulas. 
\subsection{Construction based on the sum of Bernoulli rvs}\label{ss:construction-n}
We first define the (univariate) rv $N_d = \sum_{j = 1}^{d} I_j$, with support $\{0, 1, \dots, d\}$, representing the sum of exchangeable Bernoulli rvs. The relationship between the pmf of $(I_1, \dots, I_d)$ and $N_d$ is
\begin{align*}
	\pr(N_d = k) &= \sum_{\substack{\{i_1, \dots, i_d\} \in \{0, 1\}^d \\ i_\bullet = k}} \pr\left(I_1 = i_1, \dots, I_d = i_d\right)\\
	&= \binom{d}{k} \pr(I_1 = 1, \dots, I_k = 1, I_{k+1} = 0, \dots, I_d = 0),
\end{align*}
where $i_\bullet = \sum_{j = 1}^d i_j$, the second equality follows by exchangeability of $(I_1, \dots, I_d)$. 

Let $\mathcal{N}_d$ represent the class of pmfs for univariate random variables with support $\{0, \dots, d\}$ with mean $d/2$.
%Consider the class of pmfs for univariate random variables, named $\mathcal{N}_d$, with support $\{0, \dots, d\}$ with mean $d/2$. 
In Section 3.2 of \cite{fontana2021model}, the authors show that there is a one-to-one correspondence between the class of pmfs for exchangeable Bernoulli random vectors and $\mathcal{N}_d$. This construction is useful since it identifies the joint pmf of $(I_1, \dots, I_d)$ only through the pmf of $N_d$. 

\subsection{Construction based on a vector of probabilities}

One can specify the multivariate distribution of $(I_1, \dots, I_d)$ by the vector of probabilities $(\zeta_0, \zeta_1, \dots, \zeta_d)$, where $\zeta_0 = 1$ and $\zeta_k = \Pr(I_1 = 1, \dots, I_k = 1)$, $k \in \{1,\dots,d\}$ and $d \in \{1, 2, \dots\}$. For eFGM copulas, we require $\zeta_1 = 1/2$. We recall Theorem 1 from \cite{madsen1993generalized} which provides a sufficient condition for the values of $\zeta_k, k \in \{1, \dots, d\}$. 
\begin{theorem}\label{thm:madsen}
	Let $\psi(t)$ be a completely comonotone function for $t \geq 0$. If $\zeta_k = \psi(k)$, then $\pr(N_d = k) \geq 0$, for $k \in \{0, 1, \dots, d\}$. 
\end{theorem}
The relationship between the values $(\zeta_0, \dots, \zeta_d)$, characterizing the multivariate distribution of a vector of $d$ exchangeable rvs $(I_1, \dots, I_d)$, and the values of the components of the vector of dependence parameters $(\theta_2, \dots, \theta_d)$ of the eFGM copula is established in the next result.
\begin{corollary}\label{cor:madsen}
	Let $d \in \{2, 3, \dots \}$ be fixed. For a given vector $(\zeta_0, \dots, \zeta_d)$ satisfying the conditions of Theorem \ref{thm:madsen} with $\zeta_0 = 1$, $\zeta_1 = 1/2$, we have 
	\begin{align}
		\theta_{k} &= (-2)^{k} \sum_{l = 0}^{k}\binom{k}{l} \zeta_{l}\left(-\frac{1}{2}\right)^{k-l} = \sum_{l = 0}^{k}\binom{k}{l} \zeta_{l}\left(-2\right)^{l},
		\quad k\in \{2,\ldots,d\}.\label{eq:exch-binome}
	\end{align}
\end{corollary}
\begin{proof}
	Expanding the product in \eqref{eq:theta_from_I} yields
	\begin{align*}
		\theta_k &= (-2)^k E\left\{\left(-\frac{1}{2}\right)^k + \sum_{l = 1}^k E\left(I_l\right)\left(-\frac{1}{2}\right)^{k-1} + \sum_{l = 2}^k \sum_{1 \leq j_1 < \dots < j_l \leq k} I_{j_1} \dots I_{j_l} \left(-\frac{1}{2}\right)^{k - l} \right\}.
	\end{align*}
	Since $\left(I_1, \dots, I_d\right)$ are exchangeable random variables, one has $E\left(I_{j_1}\dots I_{j_k}\right) = \zeta_k$ for all $k$-dimensional vectors $(j_1,\dots,j_k)$ such that $1 \leq j_1 < \dots < j_k \leq d$ and $k \in \{2,\dots,d\}$. 
\end{proof}
Since \eqref{eq:exch-binome} does not depend on $d$, the first $k$-dependence parameters from Corollary \ref{cor:madsen} are
\begin{equation}
	 	\theta_{k} = \begin{cases}
		 		4\zeta_2 - 1,& k = 2, ~d \geq 2\\
		 		-8\zeta_3 + 12 \zeta_2 - 2, & k = 3, ~d\geq 3\\
		 		16\zeta_4 - 32 \zeta_3 + 24\zeta_2 - 3, & k = 4, ~d \geq 4\\
		 		-32\zeta_5 + 80\zeta_4 - 80\zeta_3 + 40\zeta_2 - 4, & k = 5, ~d \geq 5
		 	\end{cases}.
\end{equation}
\begin{example}[Model 3 of \cite{madsen1993generalized}]
  Madsen considered the model $\zeta_k = \beta + (1 - \beta)\alpha^k$, for $(\alpha, \beta) \in [0, 1]^2$ and $k \in \{0, 1, \dots\}$. With the constraint that $\zeta_1 = 1/2$, we have $\alpha = (1/2-\beta)/(1 - \beta)$ for $\beta \in [0, 0.5]$ and we have one free parameter $\beta$, which is the dependence parameter. Inserting these probabilities in \eqref{eq:exch-binome} yields $\theta_k = \beta(-1)^k + (1 - \beta)(1 - (1-2\beta)/(1 - \beta))^k$, for $k \in \{0, 1, \dots\}$. The case $\beta = 0$ yields the independence copula and $\beta = 0.5$ yields the extreme positive dependence copula $C^{EPD}$ that we will describe in Theorem \ref{thm:epd}.
\end{example}
A more convenient way of specifying the values of $\zeta_k$ for $k = 0, 1, \dots$ uses Laplace-Stieltjes transforms (LST). Let us first recall Theorem 1a of Section XIII.4 in \cite{feller1968introduction}.
\begin{theorem}
	If $\psi(0) = 1$ and $\psi$ is completely comonotone, then $\psi$ is the LST of a strictly positive rv $Y$, that is, $\psi(t) = \mathcal{L}_Y(t) = E(e^{-Yt})$.
\end{theorem}
\begin{corollary}\label{corr:zeta-laplace}
	Setting $\zeta_k, k = 0, 1, \dots, d$ with $\zeta_k = \mathcal{L}_Y(rk)$, for $r > 0$ will generate probability values which satisfy the conditions of Theorem \ref{thm:madsen}. For a symmetric multivariate exchangeable Bernoulli random vector, we have $\zeta_1 = \mathcal{L}_Y(r) = 1/2$, implying $\zeta_k = \mathcal{L}_Y\left(k \times \mathcal{L}_Y^{-1}\left(1/2\right)\right)$, for $k = 1, \dots, d$. 
\end{corollary}
% \begin{example} \label{exa:definetti-gamma}
	% 	Let $Y \sim Gamma\left(1/\alpha, 1\right)$ with $\mathcal{L}_Y(t) = \left(1 + t\right)^{-1/\alpha}$ and $\mathcal{L}_Y^{-1}(u) = u^{-\alpha} - 1$. Given the conditions of Corollary \ref{corr:zeta-laplace}, one has
	% 	\begin{equation}
		% 		\zeta_k = \mathcal{L}_Y\left\{k\left((1/2)^{-\alpha} - 1\right)\right\}=\mathcal{L}_Y\left\{k\left(2^{\alpha} - 1\right)\right\}= \left(k 2^\alpha - (k-1)\right)^{-1/\alpha}.
		% 	\end{equation}
	% 	With $\alpha = 2$ and $d = 5$, the dependence parameter vector $(\theta_2, \dots, \theta_5)$ is $(0.5119, 0.0058, 0.3895, 0.0084)$.
	% \end{example}
\begin{remark}
	As shown in \cite{george1995full}, the constructions based on $N_d$ from Subsection \ref{ss:construction-n} and based on the vector of probabilities $(\zeta_0, \dots, \zeta_d)$ of the current subsection are equivalent and the following relationship holds:
	$$\pr(N_d = k) = \binom{d}{k} \sum_{l = 0}^{d-k} (-1)^l \binom{d - k}{l}\zeta_{k+l}, \quad k \in \{0, 1, \dots, d\}.$$
\end{remark}

\subsection{Construction with finite mixtures}
One can also construct exchangeable Bernoulli distributions using finite mixtures, that is, 
\begin{equation}\label{eq:bernoulli-mixture}
	\pr(I_1 = i_1, \dots, I_d = i_d) = \int_{0}^{1} \pr(I_1 = i_1, \dots, I_d = i_d \vert \Lambda = \lambda) \diff F_{\Lambda}(\lambda),
\end{equation}
where $\Lambda$ is a mixing random variable defined on $[0, 1]$. According to \eqref{eq:bernoulli-mixture}, conditional on the mixing rv $\Lambda$, $\left(I_1, \dots, I_d\right)$ are conditionally independent. One must select a distribution for $\Lambda$ such that $E\left(\Lambda\right) = 1/2$. From \eqref{eq:bernoulli-mixture}, it follows that
\begin{align}
	 \zeta_k &= \pr(I_1 = 1, \dots, I_k = 1) = \int_0^1 \lambda^k dF_{\Lambda}(\lambda) = E\left(\Lambda^k\right), \quad k \in \{0, \dots, d\};\\
	 f_{I_1, \dots, I_d}(i_1, \dots, i_d) &= \int_0^1 \lambda^{i_\bullet}(1 - \lambda)^{d - i_\bullet} \diff F_{\Lambda}(\lambda) = E\left\{\Lambda^{i_\bullet}\left(1-\Lambda\right)^{d-i_\bullet}\right\}, \quad (i_1, \dots, i_d) \in \{0, 1\}^d.\label{eq:pmf-beta}
\end{align}
For $k \in \{2,\ldots,d\}$, combining \eqref{eq:theta_from_I} and \eqref{eq:bernoulli-mixture} leads to 
\begin{equation}\label{eq:exch-esp-theta}
	\theta_k = (-2)^k E_{\Lambda}\left[E\left\{\left.\prod_{j = 1}^k \left(I_j - \frac{1}{2}\right)\right\vert \Lambda\right\} \right] = (-2)^k E\left\{\left(\Lambda - \frac{1}{2}\right)^k\right\}.
\end{equation}
\begin{remark}
	Assuming $(I_1, \dots, I_d)$ is a subsequence from an infinite sequence of positively correlated exchangeable Bernoulli trials, the famous result from de Finetti (\cite{de1929funzione}) states that there exists a rv $\Lambda$ such that \eqref{eq:bernoulli-mixture} holds. See also \cite{diaconis1977finite} for extensions when $(I_1, \dots, I_d)$ is a subsequence from a finite sequence of exchangeable Bernoulli trials.
\end{remark}
\begin{remark}
	Using the finite mixture construction and \eqref{eq:exch-esp-theta}, one can interpret the copula parameters from the mixing rv. We have $\theta_{2}\propto Var(\Lambda)$, which implies that the variance of the mixing rv induces the pairwise dependence. Then, since $\theta_{3} \propto - E[\{\Lambda - E(\Lambda)\}^3]$, we interpret $\theta_{3}$ as proportional to the negative of the skewness. If the density function of $\Lambda$ is symmetric about the mean (skewness of 0), then $\theta_{k} = 0$ when $k$ is odd. When $k$ is even, we have $E[\{\Lambda - E(\Lambda)\}^k] \geq 0$, implying $\theta_{k} \geq 0$, so the finite mixture construction induces positive dependence. 
\end{remark}

A family of distributions for $\Lambda$ will generate a family of eFGM copulas. The following example presents the beta-eFGM family of copulas with $\Lambda$ set to follow a Beta distribution. 
\begin{example}\label{exa:definetti-beta}
	Let $\Lambda \sim Beta(\alpha, \alpha)$ for $\alpha > 0$ with probability density function
	$$f_{\Lambda}(\lambda) = \frac{[\lambda(1-\lambda)]^{\alpha - 1}}{B(\alpha, \alpha)}, \quad 0\leq \lambda \leq 1,$$
	in which case $\zeta_1 = E\left(\Lambda\right) = 1/2$. The only parameter is $\alpha$, hence it acts as a dependence parameter for this family of eFGM copulas. The pmf in \eqref{eq:bernoulli-mixture} becomes
	\begin{equation}\label{eq:pmf-beta-bern}
		f_{I_1, \dots, I_d}(i_1, \dots, i_d) = \frac{\Gamma(2\alpha)}{\Gamma(\alpha)^2} \frac{\Gamma(2\alpha + d)}{\Gamma\left(\alpha + i_\bullet\right)\Gamma\left(\alpha + d - i_\bullet\right)}
		= \frac{B\left(\alpha + i_\bullet, \alpha + d - i_\bullet\right)}{B(\alpha, \alpha)},
	\end{equation}
	where $B(a, b)$ is the Beta function $\Gamma(a)\Gamma(b)/\Gamma(a+b)$, see Chapter 7 of \cite{joe1997multivariate} for details. Note that the distribution of $N_d$ when $(I_1, \dots, I_d)$ has pmf \eqref{eq:pmf-beta-bern} is called the Beta-Binomial distribution. From \eqref{eq:exch-esp-theta}, the dependence parameters are $\theta_k = B(\alpha+1/2, (k+1)/2)/B(\alpha+(k+1)/2, 1/2)$ for $k = 2, 4, 6, \dots$ and $\theta_k = 0$ for $k = 3, 5, 7, \dots$. We provide a proof in Appendix \ref{app:proof-beta-exch}. The beta-eFGM copula is
	 	\begin{equation} \label{eq:beta-copule}
		 		C\left(u_1, \dots, u_d\right) = \prod_{j = 1}^{d} u_j \left(1 + \sum_{l = 1}^{\left\lfloor \frac{d}{2} \right\rfloor}\sum_{1\leq j_{1}<\cdots <j_{2l}\leq d} \frac{B(\alpha+1/2, l + 1/2)}{B(\alpha+l+1/2, 1/2)}\overline{u}_{j_1}\cdots \overline{u}_{j_{2l}}\right),
		 	\end{equation}
	 	for $\left(u_1, \dots, u_d\right) \in [0,1]^{d}$. 
	 	We have the following representations for the dependence parameters when $k$ even:
	 	\begin{equation}
		 		\theta_{k} = \begin{cases}
			 		\frac{1}{2\alpha + 1},& k = 2, d \geq 2\\
			 		\frac{3}{(2\alpha + 1)(2\alpha + 3)}, & k = 4, d \geq 4\\
			 		\frac{15}{(2\alpha + 1)(2\alpha + 3)(2\alpha + 5)}, & k = 6, d \geq 6\\
			 	\end{cases}; \qquad \theta_{k} = \prod_{l = 1}^{k/2} \frac{2l - 1}{2\alpha + 2l - 1}, \quad k \in \{2, 4, 6, \dots\}, d \geq k.
		 	\end{equation}
	 	For $\alpha \downarrow 0$, then $\Lambda \sim Beta(\alpha, \alpha)$ converges to a discrete distribution with $\pr(\Lambda = 0) = \pr(\Lambda = 1) = 1/2$. The resulting FGM copula becomes the extreme positive dependence copula $C^{EPD}$ that we will describe in Theorem \ref{thm:epd}. When $\alpha \uparrow \infty$, then $\Lambda \sim Beta(\alpha, \alpha)$ converges to a discrete distribution with $\pr(\Lambda = 1/2) = 1$. Then, $E\left\{(\Lambda - 1/2)^k\right\} = 0$ for $k \in \{2, 3, \dots\}$ and the corresponding FGM copula is the independence copula. 
\end{example}

\section{Extreme points of exchangeable FGM copulas}\label{ss:rays}

In this section, we show that the parameters of any eFGM copula can be expressed as a convex combination of linearly independent parameters of eFGM copulas. More precisely, in any dimension $d \in \{2, 3, \dots \}$, we find convex hulls for the dependence parameters for the class of $d$-variate eFGM copulas. We call each vertex in the hull extreme points, and we seek the extreme points of $\mathcal{T}_d$, which are the extreme points from the set of inequalities in \eqref{eq:constraints-exch}. See, for example, Section 18 of \cite{rockafellar1970convex} for the relationship between convex hulls and extreme points, and \cite{terzer2009large} for details on extremal rays of convex cones. 
\begin{theorem}\label{prop:convex-exch}
	Convex combinations of eFGM copula parameters are eFGM copula parameters.
\end{theorem} 
\begin{proof}
	Consider a vector of probabilities $(\lambda_1, \dots, \lambda_{n_d})$ such that $\lambda_m \geq 0$, for $m \in \{1, \dots, n_d\}$ and $\lambda_1 + \dots + \lambda_{n_d} = 1$ with the notation that $\theta_{m, k}$ is the $k$-dependence parameter for the $m$th set of parameters. Consider the parameters $(\theta_{1,2}, \dots, \theta_{1,d}), \dots, (\theta_{n_d,2}, \dots, \theta_{n_d,d})$. A convex combination of eFGM copulas has parameters $\theta_k = \sum_{m = 1}^{n_d} \lambda_m \theta_{m, k}$, $k \in \{2, \dots, d\}$. One must then verify that the constraints \eqref{eq:constraints-exch} remain satisfied, indeed,
	\begin{align*}
		1+\sum_{k=2}^{d}\sum_{1\leq j_{1}<\dots <j_{k}\leq d}\theta_k\varepsilon _{j_{1}}\dots \varepsilon _{j_{k}}&= 1+\sum_{k=2}^{d}\sum_{1\leq j_{1}<\dots <j_{k}\leq d}\sum_{m = 1}^{n_d} \lambda_m\theta_{m, k}\varepsilon _{j_{1}}\dots \varepsilon _{j_{k}}\nonumber\\
		&=\sum_{m = 1}^{n_d} \lambda_m \left(1+\sum_{k=2}^{d}\sum_{1\leq j_{1}<\dots <j_{k}\leq d}\theta_{m, k}\varepsilon _{j_{1}}\dots \varepsilon _{j_{k}}\right),
	\end{align*}
	for $\{\varepsilon_{1},\dots, \varepsilon_{d}\} \in \{-1,1\}^d$. Since $\lambda_m \geq 0$ for $m \in \{1, \dots, n_d\}$, every summand above satisfies \eqref{eq:constraints-exch}, implying that $(\theta_2, \dots, \theta_d)$ also satisfies \eqref{eq:constraints-exch}.
\end{proof}

Let $(\lambda_1, \dots, \lambda_{n_d})$ be a vector of probabilities such that $\lambda_m \geq 0, m \in \{1, \dots, n_d\}$, and $\lambda_1 + \dots + \lambda_{n_d} = 1$. The pmfs in $\mathcal{N}_d$ can be expressed as a convex combination of $n_d$ extreme points, that is, 
\begin{equation}\label{eq:decomp-nd}
	\pr\left(N_d = k\right) = \sum_{j = 1}^{n_d} \lambda_j \pr\left(N_{jd} = k\right), \quad k = 0, \dots, d, 
\end{equation}
where $N_{jd}\in \mathcal{N}_d$ is a rv with pmf corresponding to the $j$\textsuperscript{th} extreme point, $j \in \{1, \dots, n_d\}$.
\begin{proposition}
	The pmfs of the rvs $N_{jd}$ corresponding to the extreme points of $\mathcal{N}_d$ are
	\begin{equation}\label{eq:extr-rays-pj}
		\pr(N_{jd} = k) = \begin{cases}
			\frac{j_2 - d/2}{j_2 - j_1}, & k = j_1\\
			\frac{d/2 - j_1}{j_2 - j_1}, & k = j_2\\
			0, & \text{otherwise}
		\end{cases}
	\end{equation}
	for every combination of $j_1 \in \{0, 1, \dots, j_1^\wedge\}$ and $j_2 \in \{j_2^\vee, j_2^\vee + 1, \dots, d\}$, where $(j_1^\wedge, j_2^\vee) = ((d - 1)/2,\allowbreak (d + 1)/2)$, if $d$ is odd, and $(d/2 - 1, d/2 + 1)$, if $d$ is even. If $d$ is even, the extreme points also contains the one-point distribution at $k = d/2$.
\end{proposition}
\begin{proof}
	See Proposition 4.1 of \cite{fontana2021model}.
\end{proof}

From each extreme point in \eqref{eq:extr-rays-pj}, one can then extract an associated extreme point for the dependence parameters of eFGM copulas, the solutions solve the dual problem of finding the extreme points in the set of inequalities in \eqref{eq:constraints-exch}. In other words, the extreme points of $\mathcal{N}_d$ map to extreme points of $\mathcal{T}_d$. For $d = 2$, the eFGM copula parameters corresponding to the extreme points of $\mathcal{T}_2$ are $-1$ and $1$. 
%For $d = 3$, the parameters $(\theta_2, \theta_3)$ of the four extreme points of $\mathcal{T}_3$ are $(0, 1), (-1/3, 0)$, $(1, 0)$ and $(0, -1)$. For $\mathcal{T}_4$, the copula parameters associated to the five extreme points are $(1/3, 2/3, -1/3)$, $(0, 0, -1)$, $(1, 0, 1)$, $(1/3, -2/3, -1/3)$ and $(-1/3, 0, 1)$. 

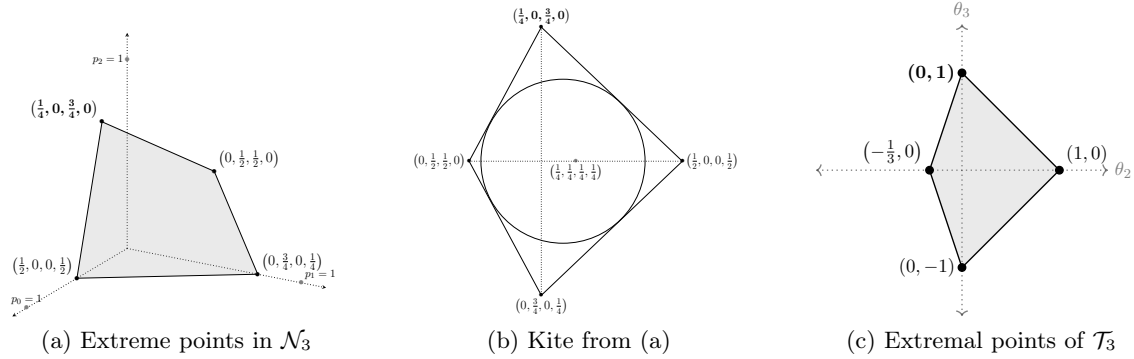
\begin{figure}[ht]
	\centering
	\subfloat[Extreme points in $\mathcal{N}_3$]{{\resizebox{0.27\textwidth}{!}{
	\begin{tikzpicture}[scale=1,tdplot_main_coords]
		
		\newdimen\R
		\R=7pt
		
		\draw[thick,-stealth, dotted] (0,0,0)--(0,0,\R+1); 
		\draw[thick,-stealth, dotted] (0,0,0)--(0,\R+1,0);
		\draw[thick,-stealth, dotted] (0,0,0)--(\R+1,0,0);
		
		%	\coordinate[label=left:{$(0,0,0)$}] (O) at (0,0,0);
		%	\coordinate[label=left:{$(0,0,1)$}] (z) at (0,\R,0);
		%	\coordinate[label=left:{$(0,1,0)$}] (y) at (0,0,\R);
		%	\coordinate[label=     {$(1,0,0)$}] (x) at (\R,0,0);
		
		%	\coordinate[label=left:{$(0,0,0)$}] (O) at (0,0,0);
		\coordinate[label=above right:{ $p_1 = 1 $}] (z) at (0,\R,0);
		\coordinate[label=left:{ $p_2 = 1$}] (y) at (0,0,\R);
		\coordinate[label=     { $p_0 = 1$}] (x) at (\R,0,0);
		
		\coordinate[label=above right:{\Large $\left(0, \frac{1}{2}, \frac{1}{2}, 0\right)$}] (p1) at (0,\R/2,\R/2);
		\coordinate[label=above left:{\Large $\boldsymbol{\left(\frac{1}{4}, 0, \frac{3}{4}, 0\right)}$}] (p2) at (\R/4,0,3\R/4);
		\coordinate[label=above left:{\Large $\left(\frac{1}{2}, 0, 0, \frac{1}{2}\right)$}] (p3) at (\R/2,0, 0);
		\coordinate[label=above right:{\Large $\left(0, \frac{3}{4}, 0, \frac{1}{4}\right)$}] (p4) at (0,3\R/4,0);
		
		%	\node[gray] at (O) {\textbullet};
		\node[gray] at (z) {\textbullet};
		\node[gray] at (x) {\textbullet};
		\node[gray] at (y) {\textbullet};
		
		\node at (p1) {\textbullet};
		\node at (p2) {\textbullet};
		\node at (p3) {\textbullet};
		\node at (p4) {\textbullet};

%		\node[gray] at (\R/4, \R/4, \R/4) {\textbullet};
		
		\fill[gray!80,opacity=0.2] (p1) -- (p2) -- (p3) -- (p4)-- cycle; 
		%	\fill[gray!80,opacity=0.2] (d1) -- (d2) -- (d4)-- cycle; 
		%	\fill[gray!80,opacity=0.2] (d1) -- (d5) -- (d3)-- cycle; 
		% draw frames
		\draw []       (p1)--(p2)--(p3)--(p4)--(p1);
		%	\draw [dashed] (p3)--(p4);
		%	\draw [dashed] (d5)--(d2);
		%	\draw []       (d5)--(d3);
		%	\draw []       (d5)--(d1);
\end{tikzpicture}}}}
	\qquad
	\subfloat[Kite from (a)]{{\resizebox{0.27\textwidth}{!}{
	\begin{tikzpicture}[scale=1,thick]
		
		\newdimen\R
		\R=10cm
		
		\coordinate[label=left:{\Large $\left(0,\frac{1}{2},\frac{1}{2},0\right)$}]  (A) at (-0.2932034\R, 0){};
		\coordinate[label=above:{\Large $\boldsymbol{\left(\frac{1}{4},0,\frac{3}{4},0\right)}$}]  (B) at (0, 0.5448624\R){};
		\coordinate[label=below:{\Large $\left(0,\frac{3}{4},0,\frac{1}{4}\right)$}]  (C) at (0, -0.5448624\R){};
		\coordinate[label=right:{\Large $\left(\frac{1}{2},0, 0,\frac{1}{2}\right)$}]  (D) at (0.572822\R, 0){};
		\coordinate[label=below:{\Large  $\left(\frac{1}{4},\frac{1}{4},\frac{1}{4},\frac{1}{4}\right)$}]  (ind) at (0.1398093\R, 0){};
		%	\coordinate[label=below:{Mid}]  (mid) at (0.088\R, 0){};
		\coordinate (mid) at (0.088\R, 0){};
		\coordinate  (E) at (0.318\R, 0.242\R){};
		
		\node at (A) {\textbullet};
		\node at (B) {\textbullet};
		\node at (C) {\textbullet};
		\node at (D) {\textbullet};
		\node[gray] at (ind) {\textbullet};
%		\node[gray] at (mid) {\textbullet};
%		\node[gray] at (E) {\textbullet};
		
		\draw[] (0.088\R, 0) circle (0.3333\R);
		
		%	\draw[dashed, color = gray] (0,-0.1363\R) circle (0.433\R);
		
		\draw []       (A)--(B)--(D)--(C)--(A);
		
		\draw [dotted] (A) -- (D);
		\draw [dotted] (B) -- (C);
		
%		\draw (mid) -- (E) node [midway, above left] {\Large $\frac{1}{3}$};
		
\end{tikzpicture}}}}
	\qquad
	\subfloat[Extremal points of $\mathcal{T}_3$]{{\resizebox{0.27\textwidth}{!}{
\begin{tikzpicture}[scale=2,axis/.style={<->,dotted, opacity=0.5},thick]
	
	\draw[axis] (-1.5, 0) -- (1.5, 0) node [right] {$\theta_2$};
	\draw[axis] (0, -1.5) -- (0, 1.5) node [above] {$\theta_3$};
	
	\coordinate  (d1) at (0, 1){};
	\coordinate  (d2) at (-1/3, 0){};
	\coordinate  (d3) at (1, 0){};
	\coordinate  (d4) at (0, -1){};
	
	% fill gray color with opacity
	\fill[gray!80,opacity=0.2] (d1) -- (d3) -- (d4) -- (d2) -- cycle; 
	% draw frames
	\draw []       (d1)--(d3)--(d4)--(d2)--(d1);
	%		\draw [dashed] (d5)--(d4);
	%		\draw [dashed] (d5)--(d2);
	%		\draw []       (d5)--(d3);
	%		\draw []       (d5)--(d1);
	
	% --- labels for vertices
	
	\draw[fill=black] (d1) circle (0.1em) node[left] {$\boldsymbol{(0, 1)}$};
	\draw[fill=black] (d2) circle (0.1em) node[above left] {$\left(-\frac{1}{3}, 0\right)$};
	\draw[fill=black] (d3) circle (0.1em) node[above right] {$(1, 0)$};
	\draw[fill=black] (d4) circle (0.1em) node[left] {$(0, -1)$};
	%		\draw[fill=black] (d5) circle (0.1em) node[above right] {(-1/3, 0, 1)};
	
	%		\fill[red!80,opacity=0.2] (-1, -1, 0) -- (-1, 1, 0) -- (1, 1, 0) -- (1, -1, 0) -- cycle; 
	
	%		\foreach \i in {1,2,...,5}
	%		{
	%			\draw[fill=black] (d\i) circle (0.1em)
	%			node[above right] {\tiny \i};
	%		}
\end{tikzpicture}
}}}
	\caption{Convex hull of admissible eFGM copula parameters for three dimensions.}
	\label{fig:hull2}
\end{figure}

Figure \ref{fig:hull2} (a) presents the coordinates $(p_0, p_1, p_2)$ from the extremal points of $\mathcal{N}_3$, the last value is not free since $p_3 = 1 - p_0 - p_1 - p_2$. The convex hull forms a geometric kite on a plane. In Figure \ref{fig:hull2} (b), we present the kite from (a) along with the kite's inscribed circle, which has radius $1/3$. In Figure \ref{fig:hull2} (c) we present the extreme points of $\mathcal{T}_3$, which is also a kite, but not a scaled version of the kite in (b). The coordinates associated to the extreme point $(\theta_2, \theta_3) = (0, 1)$ is presented in bold.

Figure \ref{fig:hull3} (a) presents the convex hull of pmfs $(p_0, \dots, p_4)$ generated by the extreme points of $\mathcal{N}_4$. Each coordinate represents a point $(p_0, p_1, p_2, p_3)$ since $p_4$ is not free; we have $p_4 = 1 - p_1 - p_2 - p_3$. We place the coordinates inside a tesseract defined by the cartesian product $[0, 1]^4$, to represent the coordinates in four dimensions.\footnote{We use the tesseract design by Claude Bragdon, see \cite{rucker2012geometry} for details.}
The convex polytope generated by the extreme points of $\mathcal{N}_4$ corresponds to a pyramid with a kite base, called a kite pyramid, with the most negative dependence case at the apex. Figure \ref{fig:hull3} (b) presents the extremal points of $\mathcal{T}_4$, with the cube $[-1, 1]^3$ in dotted lines for scale. The shape of $\mathcal{T}_4$ is also a kite pyramid. Finally, we present the dependence parameters associated to the extreme points of $\mathcal{N}_{10}$ in Table \ref{tab:extreme10}. Rows 22 and 27 (in bold) represent respectively the extremal positive and negative dependence within eFGM copulas as we will show in \eqref{eq:supermodular-u}. 

%\begin{figure}
%		\centering
%		\subfloat[$d = 3$]{{\input{hull2}}}
%		\qquad
%		\subfloat[$d = 4$]{{\input{hull3}}}
%		\caption{Convex hull of admissible eFGM copula parameters for three and four dimensions.}
%		\label{fig:hull23}
%\end{figure}

%\begin{figure}[t]
%	\centering
%	\subfloat[Extreme points of pmfs in $\mathcal{N}_3$, each vertex represents a vector of pmf values $(p_0, p_1, p_2, p_3)$]{{\input{hull2a}}}
%	\qquad
%	\subfloat[Associated copula parameters]{{\input{hull2}}}
%	\caption{Convex hull of admissible eFGM copula parameters for three dimensions.}
%	\label{fig:hull2}
%\end{figure}

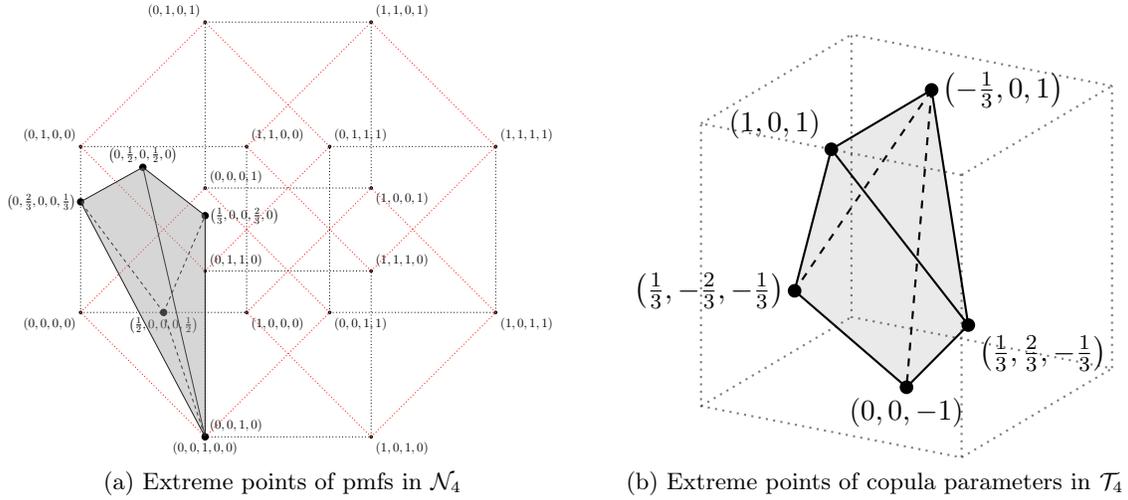
\begin{figure}[ht]
	\centering
	\subfloat[Extreme points of pmfs in $\mathcal{N}_4$]{{%\begin{tikzpicture}[scale=2,axis/.style={<->,dotted, opacity=0.5},thick]
%	
%	%	\draw[axis] (-1.5, 0) -- (1.5, 0) node [right] {$\theta_2$};
%	%	\draw[axis] (0, -1.5) -- (0, 1.5) node [above] {$\theta_3$};
%	
%	\coordinate  (d1) at (0.2, -0.8){};
%	\coordinate  (d2) at (-1.2, 0){};
%	\coordinate  (d3) at (-0.8, 0.9){};
%	\coordinate  (d4) at (0.9, 0.95){};
%	\coordinate  (d5) at (1, 0){};
%	
%	
%	% fill gray color with opacity
%	\fill[gray!80,opacity=0.2] (d1) -- (d2) -- (d3) -- (d4) -- (d5) -- cycle; 
%	% draw frames
%	\draw []       (d1)--(d2)--(d3)--(d4)--(d5)--(d1);
%	%		\draw [dashed] (d5)--(d4);
%	%		\draw [dashed] (d5)--(d2);
%	%		\draw []       (d5)--(d3);
%	%		\draw []       (d5)--(d1);
%	
%	% --- labels for vertices
%	
%	\draw[fill=black] (d1) circle (0.1em) node[below] {$\left(0, \frac{1}{2}, 0, \frac{1}{2}, 0\right)$};
%	\draw[fill=black] (d2) circle (0.1em) node[below, rotate = -45] {$\left(0, \frac{2}{3}, 0, 0, \frac{1}{3}\right)$};
%	\draw[fill=black] (d3) circle (0.1em) node[above] {$\left(\frac{1}{2}, 0, 0, 0, \frac{1}{2}\right)$};
%	\draw[fill=black] (d4) circle (0.1em) node[above] {$\left(0, 0, 1, 0, 0\right)$};
%	\draw[fill=black] (d5) circle (0.1em) node[below, rotate = 45] {$\left(\frac{1}{3}, 0, 0, \frac{2}{3}, 0\right)$};
%	%		\draw[fill=black] (d5) circle (0.1em) node[above right] {(-1/3, 0, 1)};
%	
%	%		\fill[red!80,opacity=0.2] (-1, -1, 0) -- (-1, 1, 0) -- (1, 1, 0) -- (1, -1, 0) -- cycle; 
%	
%	%		\foreach \i in {1,2,...,5}
%	%		{
%	%			\draw[fill=black] (d\i) circle (0.1em)
%	%			node[above right] {\tiny \i};
%	%		}
%\end{tikzpicture}	
\resizebox{0.45\textwidth}{!}{
\begin{tikzpicture}[scale=5,axis/.style={<->,dotted, opacity=0.5},thick]
	
	\coordinate[label=above right:{ $\left(0, 0, 1, 0\right)$}]  (d1) at (0, 0){};
	\coordinate[label=below right:{ $\left(1, 0, 1, 0\right)$}]  (d2) at (1, 0){};
	\coordinate[label=above right:{ $\left(0, 1, 1, 0\right)$}]  (d3) at (0, 1){};
	\coordinate[label=above right:{ $\left(1, 1, 1, 0\right)$}]  (d4) at (1, 1){};
	
	\draw [dotted]       (d1)--(d3)--(d4)--(d2)--(d1);
	
	\coordinate[label=below left:{ $\left(0, 0, 0, 0\right)$}]  (d5) at (-0.75, 0.75){};
	\coordinate[label=below right:{ $\left(1, 0, 0, 0\right)$}]  (d6) at (0.25, 0.75){};
	\coordinate[label=above left:{ $\left(0, 1, 0, 0\right)$}]  (d7) at (-0.75, 1.75){};
	\coordinate[label=above right:{ $\left(1, 1, 0, 0\right)$}]  (d8) at (0.25, 1.75){};
	
	\draw [dotted]       (d5)--(d7)--(d8)--(d6)--(d5);
	
	\coordinate[label=below right:{ $\left(0, 0, 1, 1\right)$}]  (d9) at (0.75, 0.75){};
	\coordinate[label=below right:{ $\left(1, 0, 1, 1\right)$}]  (d10) at (1.75, 0.75){};
	\coordinate[label=above right:{ $\left(0, 1, 1, 1\right)$}]  (d11) at (0.75, 1.75){};
	\coordinate[label=above right:{ $\left(1, 1, 1, 1\right)$}]  (d12) at (1.75, 1.75){};

	\draw [dotted]       (d9)--(d11)--(d12)--(d10)--(d9);

	\coordinate[label=above right:{ $\left(0, 0, 0, 1\right)$}]  (d13) at (0, 1.5){};
	\coordinate[label=below right:{ $\left(1, 0, 0, 1\right)$}]  (d14) at (1, 1.5){};
	\coordinate[label=above left:{ $\left(0, 1, 0, 1\right)$}]  (d15) at (0, 2.5){};
	\coordinate[label=above right:{ $\left(1, 1, 0, 1\right)$}]  (d16) at (1, 2.5){};
	
	\draw [dotted]       (d13)--(d15)--(d16)--(d14)--(d13);
	
	\foreach \i in {1,2,...,16} \draw[fill=gray] (d\i) circle (0.02em);

	\draw [color = red, dotted] (d1)--(d5);
	\draw [color = red, dotted] (d2)--(d6);
	\draw [color = red, dotted] (d3)--(d7);
	\draw [color = red, dotted] (d4)--(d8);
	
	\draw [color = red, dotted] (d1)--(d9);
	\draw [color = red, dotted] (d2)--(d10);
	\draw [color = red, dotted] (d3)--(d11);
	\draw [color = red, dotted] (d4)--(d12);
	
	\draw [color = red, dotted] (d5)--(d13);
	\draw [color = red, dotted] (d6)--(d14);
	\draw [color = red, dotted] (d7)--(d15);
	\draw [color = red, dotted] (d8)--(d16);
	
	\draw [color = red, dotted] (d9)--(d13);
	\draw [color = red, dotted] (d10)--(d14);
	\draw [color = red, dotted] (d11)--(d15);
	\draw [color = red, dotted] (d12)--(d16);
	
	\coordinate[label=below:{ $\left(0, 0, 1, 0, 0\right)$}]  (aa) at (0, 0){};
	\coordinate[label=below:{ $\left(\frac{1}{2}, 0, 0, 0, \frac{1}{2}\right)$}]  (bb) at (-0.25, 0.75){};
	\coordinate[label=left:{ $\left(0, \frac{2}{3}, 0, 0, \frac{1}{3}\right)$}]  (ee) at (-0.75, 17/12){};
	\coordinate[label=above:{ $\left(0, \frac{1}{2}, 0, \frac{1}{2}, 0\right)$}]  (dd) at (-0.375, 1.625){};
	%	\coordinate[label=above:{ $\left(0, \frac{1}{2}, 0, \frac{1}{2}, 0\right)$}]  (dd) at (-0.1666, 1.3333){};
	\coordinate[label=right:{ $\left(\frac{1}{3}, 0, 0, \frac{2}{3}, 0\right)$}]  (cc) at (0, 4/3){};
	
	\draw[fill=black] (aa) circle (0.05em);
	\draw[fill=black] (bb) circle (0.05em);
	\draw[fill=black] (ee) circle (0.05em);
	\draw[fill=black] (dd) circle (0.05em);
	\draw[fill=black] (cc) circle (0.05em);

	%	\draw [thick]       (aa) -- (bb) -- (ee) -- (aa);
	\draw [thick]       (aa) -- (ee);
	\draw [thick]       (ee) -- (dd);
	\draw [thick]       (dd) -- (cc);
	\draw [thick, dashed]       (bb) -- (cc);
	\draw [thick]       (aa) -- (cc);
	\draw [thick]       (aa) -- (dd);
	\draw [thick, dashed]       (bb) -- (aa);
	\draw [thick, dashed]       (bb) -- (ee);
	\fill[gray!80,opacity=0.4] (aa) -- (ee) -- (dd) -- (cc)-- cycle; 
\end{tikzpicture}}}}
	\qquad
	\subfloat[Extreme points of copula parameters in $\mathcal{T}_4$]{{\begin{tikzpicture}[scale=2, tdplot_main_coords,axis/.style={<->,dotted, opacity=0.5},thick]

%	\draw[axis] (-2, 0, 0) -- (2, 0, 0) node [above left] {$\theta_2$};
%	\draw[axis] (0, -1.5, 0) -- (0, 1.5, 0) node [above] {$\theta_3$};
%	\draw[axis] (0, 0, -1.5) -- (0, 0, 1.5) node [above] {$\theta_4$};
	
	\coordinate  (d1) at (1/3, 2/3, -1/3){};
	\coordinate  (d2) at (0, 0, -1){};
	\coordinate  (d3) at (1, 0, 1){};
	\coordinate  (d4) at (1/3, -2/3, -1/3){};
	\coordinate  (d5) at (-1/3, 0, 1){};
	% fill gray color with opacity
	\fill[gray!80,opacity=0.2] (d1) -- (d3) -- (d4)-- cycle; 
	\fill[gray!80,opacity=0.2] (d1) -- (d2) -- (d4)-- cycle; 
	\fill[gray!80,opacity=0.2] (d1) -- (d5) -- (d3)-- cycle; 
	% draw frames
	\draw []       (d1)--(d3)--(d4)--(d2)--(d1);
	\draw [dashed] (d5)--(d4);
	\draw [dashed] (d5)--(d2);
	\draw []       (d5)--(d3);
	\draw []       (d5)--(d1);
	
	% --- labels for vertices
	
	\draw[fill=black] (d1) circle (0.1em) node[below right] {$\left(\frac{1}{3}, \frac{2}{3}, -\frac{1}{3}\right)$};
	\draw[fill=black] (d2) circle (0.1em) node[below] {$\left(0, 0, -1\right)$};
	\draw[fill=black] (d3) circle (0.1em) node[above left] {$\left(1, 0, 1\right)$};
	\draw[fill=black] (d4) circle (0.1em) node[left] {$\left(\frac{1}{3}, -\frac{2}{3}, -\frac{1}{3}\right)$};
	\draw[fill=black] (d5) circle (0.1em) node[right] {$\left(-\frac{1}{3}, 0, 1\right)$};
	
	%		\fill[red!80,opacity=0.2] (-1, -1, 0) -- (-1, 1, 0) -- (1, 1, 0) -- (1, -1, 0) -- cycle; 
	
	%		\foreach \i in {1,2,...,5}
	%		{
		%			\draw[fill=black] (d\i) circle (0.1em)
		%			node[above right] {\tiny \i};
		%		}
		
	\coordinate  (e1) at (1, 1, 1){};	
	\coordinate  (e2) at (1, -1, 1){};	
	\coordinate  (e3) at (1, -1, -1){};	
	\coordinate  (e4) at (1, 1, -1){};	
	\coordinate  (e11) at (-1, 1, 1){};	
	\coordinate  (e12) at (-1, -1, 1){};	
	\coordinate  (e13) at (-1, -1, -1){};	
	\coordinate  (e14) at (-1, 1, -1){};	
	
	\draw [dotted, opacity = 0.5]       (e1)--(e2)--(e3)--(e4)--(e1);
	\draw [dotted, opacity = 0.5]       (e11)--(e12)--(e13)--(e14)--(e11);
	
	\draw [dotted, opacity = 0.5]       (e1)--(e11);
	\draw [dotted, opacity = 0.5]       (e2)--(e12);
	\draw [dotted, opacity = 0.5]       (e3)--(e13);
	\draw [dotted, opacity = 0.5]       (e4)--(e14);

\end{tikzpicture}}}
	\caption{Convex hull of admissible eFGM copula parameters for four dimensions.}
	\label{fig:hull3}
\end{figure}

\begin{table}[ht]
	
	\centering
	\begin{tabular}{ccccccccc}
		$\theta_2$   & $\theta_3$ &  $\theta_4$   & $\theta_5$ &   $\theta_6$   & $\theta_7$ &  $\theta_8$  & $\theta_9$ & $\theta_{10}$ \\\hline
		1/9      &    2/9     &     11/63     &    8/63    &     11/63      &    2/9     &     1/9      &     0      &       1       \\
		1/15      &    2/15    &     1/21      &   -4/105   &    -17/525     &   -2/75    &    -13/75    &   -8/25    &      3/5      \\
		1/45      &    1/15    &    -1/105     &   -4/63    &     -1/105     &    1/15    &     1/45     &     0      &       1       \\
		-1/45     &    1/45    &     -1/63     &   -2/63    &      1/35      &    1/15    &    -1/15     &   -4/15    &      1/3      \\
		-1/15     &     0      &     1/105     &     0      &     1/105      &     0      &    -1/15     &     0      &       1       \\
		1/3      &    1/3     &     5/21      &    2/7     &      1/3       &    5/21    &     5/21     &    4/7     &     -3/7      \\
		11/45     &    8/45    &     1/45      &     0      &     -1/45      &   -8/45    &    -11/45    &     0      &      -1       \\
		7/45      &    1/15    &     -1/15     &   -2/45    &      1/75      &   -1/75    &    17/225    &   12/25    &     -1/5      \\
		1/15      &     0      &     -1/15     &     0      &      1/15      &     0      &    -1/15     &     0      &      -1       \\
		-1/45     &   -1/45    &     -1/63     &    2/63    &      1/35      &   -1/15    &    -1/15     &    4/15    &      1/3      \\
		5/9      &    1/3     &      1/3      &    4/9     &      1/3       &    1/3     &     5/9      &     0      &       1       \\
		19/45     &    2/15    &     1/21      &    4/63    &    -13/105     &  -22/105   &   -29/315    &   -24/35   &      1/7      \\
		13/45     &     0      &     -1/15     &     0      &     -1/15      &     0      &    13/45     &     0      &       1       \\
		7/45      &   -1/15    &     -1/15     &    2/45    &      1/75      &    1/75    &    17/225    &   -12/25   &     -1/5      \\
		1/45      &   -1/15    &    -1/105     &    4/63    &     -1/105     &   -1/15    &     1/45     &     0      &       1       \\
		7/9      &    2/9     &      5/9      &    4/9     &      1/3       &    2/3     &     1/9      &    8/9     &     -1/9      \\
		3/5      &     0      &      1/5      &     0      &      -1/5      &     0      &     -3/5     &     0      &      -1       \\
		19/45     &   -2/15    &     1/21      &   -4/63    &    -13/105     &   22/105   &   -29/315    &   24/35    &      1/7      \\
		11/45     &   -8/45    &     1/45      &     0      &     -1/45      &    8/45    &    -11/45    &     0      &      -1       \\
		1/15      &   -2/15    &     1/21      &   4/105    &    -17/525     &    2/75    &    -13/75    &    8/25    &      3/5      \\
		\textbf{1}   & \textbf{0} &  \textbf{1}   & \textbf{0} &   \textbf{1}   & \textbf{0} &  \textbf{1}  & \textbf{0} &  \textbf{1}   \\
		7/9      &    -2/9    &      5/9      &    -4/9    &      1/3       &    -2/3    &     1/9      &    -8/9    &     -1/9      \\
		5/9      &    -1/3    &      1/3      &    -4/9    &      1/3       &    -1/3    &     5/9      &     0      &       1       \\
		1/3      &    -1/3    &     5/21      &    -2/7    &      1/3       &   -5/21    &     5/21     &    -4/7    &     -3/7      \\
		1/9      &    -2/9    &     11/63     &   -8/63    &     11/63      &    -2/9    &     1/9      &     0      &       1       \\
		\textbf{-1/9} & \textbf{0} & \textbf{1/21} & \textbf{0} & \textbf{-1/21} & \textbf{0} & \textbf{1/9} & \textbf{0} &  \textbf{-1}
	\end{tabular}
	\caption{Copula parameters associated to the extreme points of $\mathcal{N}_{10}$.}\label{tab:extreme10}
\end{table}

\begin{corollary}\label{cor:number-rays}
	There are $(d + 1)^2/4$ extreme points if $d$ is odd and $d^2/4 + 1$ if $d$ is even. 
\end{corollary}
\begin{proof}
	See Corollary 4.6 of \cite{fontana2021model}.
\end{proof}

\begin{remark}\label{rem:non-unique}
	The representation as a convex combination of extreme points is not unique. The dependence parameters associated with extreme points of $\mathcal{N}_d$ are linearly independent but there may be different sets of weights $\lambda_1, \dots, \lambda_{n_d}$ which yield the same dependence parameter vector. For example, let $d = 3$, then one recovers the independence parameter vector $(0, 0)$ with $2^{-1} (0, 1) + 2^{-1}(0, -1)$ and $3/4(-1/3, 0) +4^{-1}(1, 0)$.
\end{remark}

\section{Dependence ordering}\label{sec:dependence-ordering}

\subsection{Supermodular order}

We aim to compare vectors of rvs, $(V_{1},\ldots,V_{d})$ and $(V_{1}^{\prime},\ldots,V_{d}^{^{\prime}})$, where, for each $j \in \{1,\ldots,d\}$, $V_{j}$ and $V_{j}'$ have the same marginal distribution. Given this condition on the marginals, we rely on dependence stochastic orders. An important dependence order is the supermodular order, see Sections 3.8 and 3.9 of \cite{muller2002comparison} for details.
\begin{definition}[Supermodular order] \label{defSupermodularOrder}
	We say $(V_1, \dots, V_d)$ is smaller than $(V_1', \dots, V_d')$ under the supermodular order,	denoted $(V_1, \dots, V_d)\preceq _{sm}(V_1', \dots, V_d')$, if $E\left\{\phi (V_1, \dots, V_d)\right\} \leq E\left\{ \phi (V_1', \dots, V_d')\right\} $ for all supermodular functions $\phi $, given that the
	expectations exist. A function $\phi :\mathbb{R}^{d}\rightarrow \mathbb{R}$ is said to be supermodular if
	\begin{eqnarray*}
		&&\phi (x_{1},\ldots,x_{i}+\varepsilon ,\ldots,x_{j}+\delta ,\ldots,x_{d})-\phi
		(x_{1},\ldots,x_{i}+\varepsilon ,\ldots,x_{j},\ldots,x_{d}) \\
		&\geq &\phi (x_{1},\ldots,x_{i},\ldots,x_{j}+\delta ,\ldots,x_{d})-\phi
		(x_{1},\ldots,x_{i},\ldots,x_{j},\ldots,x_{d})
	\end{eqnarray*}
	holds for all $(x_1, \dots, x_d)\in \mathbb{R}^{d}$, $1\leq
	i\leq j\leq d$\ and all $\varepsilon$, $\delta >0$.
\end{definition}
The supermodular order satisfies the nine desired properties for dependence orders as mentioned in Section 3.8 of \cite{muller2002comparison}. Ordering random vectors according to the supermodular order is desirable since it implies stochastic ordering results for the sum of vectors of rvs. See also \cite{shaked2007stochastic} and \cite{denuit2006actuarial} for more details on the supermodular order. 

\subsection{Supermodular ordering within eFGM copulas}\label{ss:supermodular-efgm}

The following theorem from \cite{blier-wong2022stochasticb} presents the general result for supermodular orders within FGM copulas with the stochastic representation.
\begin{theorem}\label{thm:supermodular-i-u}
	If $(I_1, \dots, I_d) \preceq_{sm} (I_1', \dots, I_d')$, then $(U_1, \dots, U_d) \preceq_{sm} (U_1', \dots, U_d')$.
\end{theorem}
\begin{definition}[Convex order]
	Let $X$ and $X'$ be random variables with finite means. We say that $X$ is smaller than $X'$ under the convex order, noted $X\preceq_{cx}X'$, if $E\{\varphi(X)\} \leq E\{\varphi(X')\}$ for all real convex functions $\varphi$ such that the expectations exist.
\end{definition}
\begin{proposition}\label{prop:order-mixture}
	Consider two random vectors $(I_1, \dots, I_d)$ and $(I_1', \dots, I_d')$ with pmfs constructed with finite mixtures as in \eqref{eq:bernoulli-mixture} with respective mixing distributions $\Lambda$ and $\Lambda'$. If $\Lambda \preceq_{cx} \Lambda'$, we have $(I_1, \dots, I_d) \preceq_{sm} (I_1', \dots, I_d')$ and the random vectors constructed with the stochastic representation FGM copulas are ordered with $(U_1, \dots, U_d) \preceq_{sm} (U_1', \dots, U_d')$.
\end{proposition}
\begin{proof}
	If $\Lambda \preceq_{cx} \Lambda'$, Proposition 4.1.iii of \cite{denuit2008comparison} or Theorem 2.11 of \cite{cousin2008comparison} implies the first result that $(I_1, \dots, I_d) \preceq_{sm} (I_1', \dots, I_d')$. Then, the second result follows from Theorem \ref{thm:supermodular-i-u}.
\end{proof}
\begin{example}
	Let $Y \sim Gamma(1/\alpha, 1/\alpha)$ and $Y' \sim Gamma(1/\alpha', 1/\alpha')$, with $0\leq \alpha \leq \alpha'$. Let also $\Lambda = \exp\{-Y k\}$ and $\Lambda' = \exp\{-Y' k'\}$, where $k = \mathcal{L}_{Y}^{-1}(0.5)$ and $k' = \mathcal{L}_{Y'}^{-1}(0.5)$. One may show that $\Lambda \preceq_{cx} \Lambda'$. Constructing $(I_1, \dots, I_d)$ and $(I_1', \dots, I_d')$ with the representation in \eqref{eq:bernoulli-mixture} and respective mixing rvs $\Lambda$ and $\Lambda'$, we have $(I_1, \dots, I_d) \preceq_{sm} (I_1', \dots, I_d')$ and $(U_1, \dots, U_d) \preceq_{sm} (U_1', \dots, U_d')$.
\end{example}
\begin{example}
	Let $\Lambda \sim Beta(\alpha, \alpha)$ and $\Lambda' \sim Beta(\alpha', \alpha')$, then one may show that $\Lambda \preceq_{cx}\Lambda'$ when $0 \leq \alpha' \leq \alpha$. When using the representation in \eqref{eq:bernoulli-mixture} with a Beta rv, decreasing $\alpha$ increases dependence. 
\end{example}

In the remainder of this section, we aim to identify the lower and upper bounds under the supermodular order for eFGM copulas.
\begin{theorem}
	Let $(I_1^-, \dots, I_d^-)$ have pmf
	\begin{equation}\label{eq:end-fgm}
		\pr(I_1^- = i_1, \dots, I_d^- = i_d) = 
		\begin{cases}
			(r + 1 - d/2) \binom{d}{r}^{-1}, & i_\bullet = r\\
			(d/2 - r) \binom{d}{r + 1}^{-1}, & i_\bullet = r + 1\\
			0, & \text{otherwise}
		\end{cases},
	\end{equation}
	where $r \leq d/2 \leq r + 1$. Also define $(I_1^+, \dots, I_d^+)$ as the vector of rvs with pmf
	\begin{equation}\label{eq:i-como}
		\pr(I_1^+ = i_1, \dots, I_d^+ = i_d) = \begin{cases}
			1/2, & i_\bullet \in \{0, d\}\\
			0, & \text{otherwise}
		\end{cases}.
	\end{equation}
	For all vectors of exchangeable symmetric Bernoulli rvs $(I_1, \dots, I_d)$, we have
	$$(I_1^-, \dots, I_d^-) \preceq_{sm} (I_1, \dots, I_d) \preceq_{sm} (I_1^+, \dots, I_d^+).$$
\end{theorem}
\begin{proof}
	The relationship in \eqref{eq:end-fgm} follows from (7.22) of \cite{joe1997multivariate} with $\pi = 1/2$, who identifies this pmf as the most negative dependence for exchangeable Bernoulli rvs. Theorem 7 of \cite{frostig2001comparison} shows that \eqref{eq:end-fgm} is also the lower bound under the supermodular order. The joint pmf in \eqref{eq:i-como} corresponds to the joint pmf when the exchangeable and symmetric Bernoulli random variables are comonotonic, which also coincides with the pmf derived from the Fréchet-Hoeffding upper bound with symmetric Bernoulli marginals. 	
\end{proof}

We note respectively the extreme negative dependence (END) and extreme positive dependence (EPD) eFGM copulas such that the following holds:
\begin{equation}\label{eq:supermodular-u}
	(U_1^{END}, \dots, U_d^{END}) \preceq_{sm} (U_1, \dots, U_d) \preceq_{sm} (U_1^{EPD}, \dots, U_d^{EPD}),
\end{equation}
for all random vectors $(U_1, \dots, U_d)$ with cdf $F_{U_1, \dots, U_d} = C$ in the eFGM family of copulas. The EPD copula, provided in Theorem 5 of \cite{blier-wong2022stochasticb}, is recalled in the following theorem.
\begin{theorem}\label{thm:epd}
	The FGM copula associated to the vector of comonotonic rvs $(I_1^+, \dots, I_d^+)$ is the EPD copula $C^{EPD}$. The dependence parameters are $\theta_k = (1 + (-1)^k)/2$, or $\theta_k = 1$, when $k$ is even, and $\theta_k = 0$, when $k$ is odd.
	 	The expression of the EPD copula $C^{EPD}$ is given by 
	 	\begin{equation} \label{eq:epd}
		 		C^{EPD}\left(u_1, \dots, u_d\right) = \prod_{j = 1}^{d} u_j \left(1 + \sum_{k = 1}^{\left\lfloor \frac{d}{2} \right\rfloor}\sum_{1\leq j_{1}<\cdots <j_{2k}\leq d} \overline{u}_{j_1}\cdots \overline{u}_{j_{2k}}\right), \quad \left(u_1, \dots, u_d\right) \in [0,1]^{d},
		 	\end{equation}
	 	where $\lfloor y \rfloor$ is the floor function returning the greatest integer smaller or equal to $y$. 
\end{theorem}
One obtains the extreme positive dependence point for the case $j_1 = 0$ and $j_2 = d$ in \eqref{eq:extr-rays-pj}. 

The lower bound for multivariate symmetric Bernoulli rvs under the supermodular order is defined in the following theorem; the proof is provided in Appendix \ref{app:lower-bound}. 
\begin{theorem}\label{thm:lower-bound-i}
	The copula constructed with the vector of rvs $\left(I_1^-, \dots, I_d^-\right)$ is the END copula, noted $C^{END}$. The dependence parameters $(\theta_2, \dots, \theta_d)$ for the END copula $C^{END}$ are given by
	\begin{equation}\label{eq:end-param}
		\theta_k = {}_2F_1\left(-\left\lfloor \frac{d + 1}{2}\right\rfloor, -k, 2\left\lfloor \frac{d + 1}{2}\right\rfloor, 2\right) = \frac{(1 + (-1)^k)}{2}\frac{\Gamma(k+1)\Gamma\left(\frac{1}{2} -\left\lfloor  \frac{d + 1}{2}\right\rfloor\right)}{2^k \Gamma\left(\frac{k}{2}+1\right) \Gamma\left(\frac{k+1}{2} - \left\lfloor \frac{d + 1}{2}\right\rfloor\right)}.
	\end{equation}
\end{theorem}

\begin{corollary}
	Alternate representations for $k$-dependence parameters for $d$ odd are 
	$$\theta_k = \begin{cases}
		-\frac{1}{d}, & k = 2, d \geq 3\\
		\frac{3}{(d-2)d}, & k = 4, d \geq 7\\
		-\frac{15}{(d-4)(d-2)d}, & k = 6, d \geq 11 
	\end{cases}, \qquad \theta_k = \prod_{l = 1}^{k/2} \frac{1 - 2l}{d-2l+2}, d \geq 2k - 1.$$
	while for $d$ even, the $k$-dependence parameters are 
	$$\theta_k = \begin{cases}
	-\frac{1}{d-1}, & k = 2, d \geq 4\\
	\frac{3}{(d-3)(d-1)}, & k = 4, d \geq 8\\
	-\frac{15}{(d-5)(d-3)(d-1)}, & k = 6, d \geq 12
	\end{cases}, \qquad \theta_k = \prod_{l = 1}^{k/2} \frac{1 - 2l}{d-2l+1}, d \geq 2k.$$
\end{corollary}

\begin{remark}
	The random vector $(I_1^-, \dots, I_d^-)$ in Theorem \ref{thm:lower-bound-i} corresponds to a dependence structure called complete mixability, see \cite{puccetti2015extremal} for details. 
\end{remark}
\begin{remark}
	Table \ref{tab:end-parameters} presents the values of $\theta_2, \dots, \theta_d$ of the END copula, $C^{END}$, for $d \in \{2, \dots, 12\}$. Although the parameters follow some pattern, it isn't obvious from looking at the values that this corresponds to the dependence parameters inducing the most negative dependence within FGM copulas. We offer a few observations on the patterns exhibited by the parameters of the END copula. First, \eqref{eq:end-param} gives the same values for consecutive values of $d \in \{3, 5, 7, \dots\}$ and $(d + 1) \in \{4, 6, 8, \dots\}$. Then, we always have $\theta_k = 0$ for $k$ odd. One also notices alternate signs for $k$ even, that is, $\theta_k$ is negative for $k/2 \in \{1, 3, 5, \dots\}$ and positive for $k/2 \in \{2, 4, 6, \dots\}$. Since $\theta_0 = 1$ and $\theta_1 = 0$ for every FGM copula, one remarks that the magnitude of the END dependence parameters $\theta_k$ is symmetric, decreasing for $k < d/2$ and increasing again for $k > d/2$. Also, $\theta_k \to 0$ as $d \to \infty$ for $k \neq d$, and the $k$-dependence parameters depend on $d$. 
\end{remark}

\begin{table}[ht]
	\centering
	\begin{tabular}{rrrrrrrrrrrr}
		$d$	& $\theta_2$ & $\theta_3$ & $\theta_4$ & $\theta_5$ & $\theta_6$ & $\theta_7$ & $\theta_8$ & $\theta_9$ & $\theta_{10}$ & $\theta_{11}$ & $\theta_{12}$ \\\hline 
		2 &         -1 &            &            &            &            &            &            &            &               &               &               \\
		3 &       -1/3 &          0 &            &            &            &            &            &            &               &               &               \\
		4 &       -1/3 &          0 &          1 &            &            &            &            &            &               &               &               \\
		5 &       -1/5 &          0 &        1/5 &          0 &            &            &            &            &               &               &               \\
		6 &       -1/5 &          0 &        1/5 &          0 &         -1 &            &            &            &               &               &               \\
		7 &       -1/7 &          0 &       3/35 &          0 &       -1/7 &          0 &            &            &               &               &               \\
		8 &       -1/7 &          0 &       3/35 &          0 &       -1/7 &          0 &          1 &            &               &               &               \\
		9 &       -1/9 &          0 &       1/21 &          0 &      -1/21 &          0 &        1/9 &          0 &               &               &               \\
		10 &       -1/9 &          0 &       1/21 &          0 &      -1/21 &          0 &        1/9 &          0 &           -1 &               &               \\
		11 &      -1/11 &          0 &       1/33 &          0 &     -5/231 &          0 &       1/33 &          0 &         -1/11 &             0 &               \\
		12 &      -1/11 &          0 &       1/33 &          0 &     -5/231 &          0 &       1/33 &          0 &         -1/11 &             0 &             1 \\
	\end{tabular}
	\caption{Extreme negative dependence copula parameters.}\label{tab:end-parameters}
\end{table}

\begin{remark}
	The bounds in \eqref{eq:supermodular-u} are valid for all FGM copulas, not only exchangeable ones. For all supermodular functions $\phi$ and random vectors $(U_1, \dots, U_d)$ with cdf $F_{U_1, \dots, U_d} = C$, a FGM copula as in \eqref{fgmco1205}, we have
	$$E\left\{\phi (U_1^{END}, \dots, U_d^{END})\right\} \leq E\left\{ \phi (U_1, \dots, U_d)\right\}\leq E\left\{ \phi (U_1^{EDP}, \dots, U_d^{EDP})\right\}.$$
	% for all supermodular functions $\phi$ and all random vectors $(U_1, \dots, U_d)$ with FGM dependence.
\end{remark}
\begin{remark}
	The term $(1 + (-1)^k)/2$ in \eqref{eq:end-param} implies that $\theta_k = 0$ for $k \in \{3, 5, 7, \dots\}$, which is also the case for the EPD copula. As noted in \cite{blier-wong2022stochasticb}, the dependence parameters for odd indices do not contribute to the overall strength of dependence.
\end{remark}

\section{Sampling and estimation}\label{sec:sample-estimation}

\subsection{Sampling}

In \cite{blier-wong2022stochasticb}, an efficient stochastic sampling method is proposed based on the stochastic representation of FGM copulas. In Algorithm \ref{algo:sampling}, we leverage the representation based on the class $\mathcal{N}_d$ from Subsection \ref{ss:construction-n} to efficiently sample observations from eFGM copulas. Note that when the pmf of $N_d$ is an extreme point of $\mathcal{N}_d$, sampling is faster since the vector of probabilities $(p_0, \dots, p_d)$ will have at most two non-zero values. Also, for subfamilies of eFGM copulas based on finite mixtures as in \eqref{eq:bernoulli-mixture}, one may sample $\tilde{N}_d$ from line 1 of Algorithm \ref{algo:sampling} by first sampling $\tilde{\Lambda}$, then sampling $\tilde{N}_d$ from a binomial distribution with $d$ trials and success probability $\tilde{\Lambda}$. 

%While the stochastic sampling method is preferable to the conditional sampling method, it may be infeasible in very high dimensions. The method requires the computation of the pmf of $(I_1, \dots, I_d)$ for every combination of $\{0, 1\}^d$. Then, one samples from the pmf of a discrete distribution with $2^d$ elements. While the sampling itself isn't an issue for high dimensions, the mere storage of the table of $2^d$ probabilities is. 

%\begin{algorithm}
%	
%	%\vspace*{-12pt}
%	\begin{tabbing}
%		\qquad \enspace 
%		\qquad \enspace 
%		\qquad \enspace For $j \in \{1, \dots, d\}$: \\
%		\qquad \qquad 
%		\qquad \qquad Compute $U_j = \Tilde{V}_0^{1/2} \Tilde{V}_1^{\Tilde{I}_j}$\\
%		\qquad \enspace 
%	\end{tabbing}\label{algo:sampling}
%\end{algorithm}

\begin{algorithm}
	\KwIn{Vector of probabilities $(p_0, \dots, p_d)$}
	\KwOut{Sample vector $(U_1, \dots, U_d)$}
	\nl Sample $\Tilde{N_d}$ from the vector of probabilities $(p_0, \dots, p_d)$\;
	\nl Sample $(\Tilde{I}_1, \dots, \Tilde{I}_d)$ with a random permutation on vector of $\Tilde{N_d}$ ones and $d - \Tilde{N_d}$ zeroes\;
	\nl \For{$j = 1, \dots, d$}{
		\nl Sample $\Tilde{V}_0, \Tilde{V}_1 \sim Unif(0, 1)$\;
		\nl Compute $U_j = \Tilde{V}_0^{1/2} \Tilde{V}_1^{\Tilde{I}_j}$\;
	}
	\nl Output $(U_1, \dots, U_d)$.
	\caption{Stochastic sampling method for eFGM copulas. } \label{algo:sampling}
\end{algorithm}

\subsection{Estimation difficulties with the FGM family of copulas}

The main difficulty in estimating the parameters of an FGM copula is that they must respect the constraints \eqref{eq:constraints-exch}. For this reason, the method of moments is unlikely to provide a set of parameters that satisfy the $2^d$ constraints. An attempt to estimate the parameters of FGM copulas is provided in \cite{ota2021effective}, by estimating the parameters one at a time and using the simplex algorithm to constrain the valid parameter set after each parameter is estimated. However, this method does not scale well to high dimensions and will provide different parameter values if the order of parameter estimation changes. In the following subsection, we provide an algorithm that guarantees that the resulting parameters satisfy \eqref{eq:constraints-exch} due to Proposition \ref{prop:convex-exch}.

\subsection{Maximum likelihood estimation}

The likelihood of a set of $m_{obs}$ observations $\left(u_{1m}, \dots, u_{dm}\right), m \in \{1, \dots, m_{obs}\}$,
% \todo[color = red!40]{notation pas joli}
% for the natural parametrization of eFGM copulas is
% \begin{equation}\label{eq:like1}
	% L(\theta_2, \dots, \theta_d) = \prod_{m = 1}^{m_{obs}}\left(1 + \sum_{k = 2}^d \sum_{1\leq j_1< \dots < j_k\leq d} \theta_k (1 - 2u_{mj_1}) \dots (1 - 2u_{mj_k})\right)
	% \end{equation}
% but the likelihood in \eqref{eq:like1} does not have a convenient shape for estimation. Instead, we use the likelihood based on 
for the stochastic representation of eFGM copulas is
\begin{equation}\label{eq:likelihood-1}
	L(\theta_2, \dots, \theta_d) = \prod_{m = 1}^{m_{obs}}\sum_{\{i_1, \dots, i_d\}\in \{0,1\}^d} f_{I_1, \dots, I_d}(i_1, \dots, i_d) \prod\limits_{l=1}^{d} \left[1 + (-1)^{i_l}(1 - 2u_{ml})\right].
\end{equation}
Maximizing \eqref{eq:likelihood-1} is feasible, but is computationally inconvenient since one needs to apply the system of constraints in \eqref{eq:constraints-exch}. Another approach involves using the representation in $\mathcal{N}_d$, estimating the parameters $p_k, k \in \{0, \dots, d\}$ under the constraints $\sum_{k = 0}^d p_k = 1, \sum_{k = 0}^d kp_k = d/2$ and $p_k \geq 0, k \in \{0, \dots, d\}$. This representation estimates $d-1$ parameters and admits a unique solution but we have not found an efficient algorithm to perform this optimization.

Instead, we use the construction based on Section \ref{ss:rays}, which defines eFGM copula parameters as a convex combination of parameters from extreme points of the pmfs in $\mathcal{N}_d$. The main advantage of this construction is that the likelihood is expressed as a finite mixture of $n_d$ points and we can use an Expectation-Maximisation approach to optimize the likelihood. The disadvantage of this approach is that the solution using the convex combinations of extreme points is not unique, as stated in Remark \ref{rem:non-unique}. Using the representation of $\mathcal{N}_d$ from Section \ref{ss:construction-n}, the likelihood is
\begin{equation}\label{eq:likelihood-2}
	\prod_{m = 1}^{m_{obs}}\sum_{k = 0}^{d}\pr(N_d = k)\frac{1}{\binom{d}{k}} \sum_{\substack{\{i_1, \dots, i_d\} \in \{0, 1\}^d \\ i_\bullet = k}}\prod\limits_{l=1}^{d} \left[1 + (-1)^{i_l}(1 - 2u_{ml})\right].
\end{equation}
Replacing \eqref{eq:decomp-nd} into \eqref{eq:likelihood-2}, the likelihood becomes
\begin{equation}\label{eq:likelihood-3}
	\prod_{m = 1}^{m_{obs}}\sum_{k = 0}^{d}\sum_{j = 1}^{n_d} \lambda_j \pr\left(N_{jd} = k\right)\frac{1}{\binom{d}{k}} \sum_{\substack{\{i_1, \dots, i_d\} \in \{0, 1\}^d \\ i_\bullet = k}}\prod\limits_{l=1}^{d} \left[1 + (-1)^{i_l}(1 - 2u_{ml})\right].
\end{equation}
Rearranging \eqref{eq:likelihood-3} yields
\begin{equation}
	\prod_{m = 1}^{m_{obs}}\sum_{j = 1}^{n_d} \lambda_j \sum_{k = 0}^{d} \pr\left(N_{jd} = k\right)\frac{1}{\binom{d}{k}} \sum_{\substack{\{i_1, \dots, i_d\} \in \{0, 1\}^d \\ i_\bullet = k}}\prod\limits_{l=1}^{d} \left[1 + (-1)^{i_l}(1 - 2u_{ml})\right] = \prod_{m = 1}^{m_{obs}}\sum_{j = 1}^{n_d} \lambda_j \xi_{mj},
\end{equation}
where
\begin{equation}\label{eq:xi_mj}
	\xi_{mj} = \sum_{k = 0}^{d} \pr\left(N_{jd} = k\right)\frac{1}{\binom{d}{k}} \sum_{\substack{\{i_1, \dots, i_d\} \in \{0, 1\}^d \\ i_\bullet = k}}\prod\limits_{l=1}^{d} \left[1 + (-1)^{i_l}(1 - 2u_{ml})\right],
\end{equation}
and does not depend on the parameters $\lambda_j, j \in \{1, \dots, n_d\}$, so can be computed once at the beginning of the optimization procedure. Using Lagrange multipliers to impose constraints on the parameters $\lambda_{j}, j = \{1, \dots, n_d\}$, the log-likelihood to maximize is
\begin{equation}\label{eq:obj-function}
	\mathcal{J}(\lambda_1, \dots, \lambda_{n_d}, \mu) = \sum_{m = 1}^{m_{obs}}\ln \left(\sum_{j = 1}^{n_d} \lambda_j \xi_{mj}\right) + \mu \left(\sum_{j = 1}^{n_d}\lambda_j - 1\right).
\end{equation}
We find the Lagrange multiplier $\mu = -m_{obs}$ and 
\begin{equation}
	\sum_{j = 1}^{m_{obs}} \frac{\xi_{jt}}{\sum_{l = 1}^{n_d}\hat{\lambda}_j \xi_{jl}} = m_{obs} \Longrightarrow \hat{\lambda}_{t} = \frac{\sum_{j = 1}^{m_{obs}}\frac{\hat{\lambda}_t \xi_{jt}}{\sum_{l = 1}^{n_d} \hat{\lambda}_l \xi_{jl}}}{m_{obs}}, \quad t \in \{0, \dots, n_d\}.
\end{equation}
We propose the following iterative algorithm to estimate the weights $\hat{\lambda}_1, \dots, \hat{\lambda}_{n_d}$.
\begin{algorithm}
	\KwIn{Number of simulations $n$, pmf $f_{\boldsymbol{I}}$}
	\KwOut{Set of simulations}
	\nl Initialize $\lambda_t^{(0)} = 1/n_d$ for $t = 1, \dots, n_d$\;
	\nl Set $l = 0$\;
	\nl \Repeat{Convergence of \eqref{eq:likelihood-3}}{
	\nl \For{$t = 0, \dots, n_d$}{
		\nl Set $\displaystyle \lambda_t^{(l+1)} = \frac{\sum_{j = 1}^{N} \frac{\lambda_t^{(l)}\xi_{jt}}{\sum_{k = 0}^{d} \lambda_t^{(l)} \xi_{jk}}}{N}$\;
		\nl Set $l = l + 1$
	}}
%	\nl \For{$l = 1, \dots, n$}{
%		\nl \For{$t = 0, \dots, n_d$}{
%			\nl set $\displaystyle \lambda_t^{(l+1)} = \frac{\sum_{j = 1}^{N} \frac{\lambda_t^{(l)}\xi_{jt}}{\sum_{k = 0}^{d} \lambda_t^{(l)} \xi_{jk}}}{N}$\;
%		}
%		\nl If convergence of \eqref{eq:likelihood-3}, stop\;
%	}
	\nl Output most recent weights $\lambda_t^{(l)}, t = 1, \dots, n_d$.
	\caption{MLE estimation as a combination of extreme points} \label{algo:em-estimation}
\end{algorithm}

\subsection{Simulation study}

To illustrate the estimation procedure, we perform a simulation study and attempt to estimate the corresponding eFGM parameters. We use Algorithm \ref{algo:sampling} to sample observations and Algorithm \ref{algo:em-estimation} to estimate the parameters $\lambda_j, j \in \{1, \dots, n_d\}$. However, we compare the resulting values of $\theta_k, k \in \{2, \dots, d\}$ since the representation based on convex combinations of extreme points is not unique, see Remark \ref{rem:non-unique}.
% (there are $d-1$ dependence parameters but we estimate $n_d$ parameters; since $n_d >d$, the representation as convex combinations is overspecified).

In this study, we consider estimation based on known uniform margins; in cases with unknown margins, one should compute pseudo-observations based on the ranks of the empirical distribution function (using the semiparametric method of \cite{genest1993statistical} or information from margins of \cite{joe1996estimation}). We consider dimension $d = 10$ and sample multivariate observations $(u_{1m}, \dots, u_{10m})$ for $m \in \{1, \dots, 10~000\}$. We then estimate the parameters $\theta_k, k \in \{2, \dots, d\}$. To the best of our knowledge, this is the first study estimating the parameters of FGM copulas for 10 dimensions, since both the stochastic representation and the assumption of exchangeability simplifies the parameter space.

We repeat the simulation and estimation 100 times with the same set of randomly generated parameters (but satisfying \eqref{eq:constraints-exch}) and present the results in Figure \ref{fig:boxplot-zoom}. The dashed line represents the true parameter, and the boxplot presents the range of estimates across the 100 replications. We present the main estimation diagnostic statistics in Table \ref{tab:simulation-study-stats}.
We note that the admissible range of individual parameters are $[-1, 1]$, and the height of the boxplots in Figure \ref{fig:boxplot-zoom} are much smaller than the admissible range. Even if the true value of parameters is close to zero, there is little variation in the parameter estimates. For example, the value of $\theta_2$ is $0.0667$, which induces weak dependence (Spearman's correlation coefficient between each pair of marginals, that is, $\rho_{S}(U_{j_1}, U_{j_2})$ for $1 \leq j_1 < j_2 \leq d$, is only $0.0667/3$), but the interquartile range is only 0.006, making the estimates significantly different than 0, on an empirical basis. Only the parameter $\theta_{10}$ has a real value outside of the interquartile range of estimated values. This isn't surprising: estimation of $\theta_5$ is based on $10!/5!/5! = 252$ different 5-tuples for each observation, while $\theta_{10}$ is based on a single 10-tuple. However, as discussed in \cite{blier-wong2022stochasticb}, $k$-dependence parameters for $k$ close to $d$ has less impact on the overall dependence: for the multivariate extensions of Spearman's rho presented in \cite{nelsen1996nonparametric}, the contribution of $10$-dependence parameters is $1/3^{10}$, while it is $1/3^2$ for $2$-dependence parameters. We conclude that the most important parameters are adequately estimated. 
\begin{table}[ht]
	
	\centering
	\begin{tabular}{rrrrrrrrrr}
		& $\theta_2$ & $\theta_3$ & $\theta_4$ & $\theta_5$ & $\theta_6$ & $\theta_7$ & $\theta_8$ & $\theta_9$ & $\theta_{10}$ \\
		Real parameter &     0.0667 &     0.1407 &     0.0709 &     0.0085 &     0.0442 &     0.0874 &    -0.0133 &    -0.1067 &        0.8667 \\
		1st Quadrant &     0.0640 &     0.1362 &     0.0669 &     0.0013 &     0.0274 &     0.0543 &    -0.0566 &    -0.1669 &        0.6840 \\
		Median &     0.0676 &     0.1390 &     0.0709 &     0.0089 &     0.0406 &     0.0716 &    -0.0307 &    -0.1278 &        0.7443 \\
		Mean &     0.0671 &     0.1392 &     0.0718 &     0.0091 &     0.0391 &     0.0722 &    -0.0305 &    -0.1288 &        0.7439 \\
		3rd Quadrant &     0.0699 &     0.1431 &     0.0759 &     0.0150 &     0.0509 &     0.0927 &    -0.0006 &    -0.0915 &        0.8211 \\
		Interquartile range &     0.0060 &     0.0069 &     0.0090 &     0.0137 &     0.0235 &     0.0383 &     0.0560 &     0.0755 &        0.1371 \\
		Standard deviation &     0.0052 &     0.0052 &     0.0071 &     0.0106 &     0.0174 &     0.0266 &     0.0373 &     0.0527 &        0.1031
	\end{tabular}
	\caption{Estimation statistics for the simulation study.}\label{tab:simulation-study-stats}
\end{table}

\begin{figure}[ht]
	\centering
	\includegraphics[width = 0.5\textwidth]{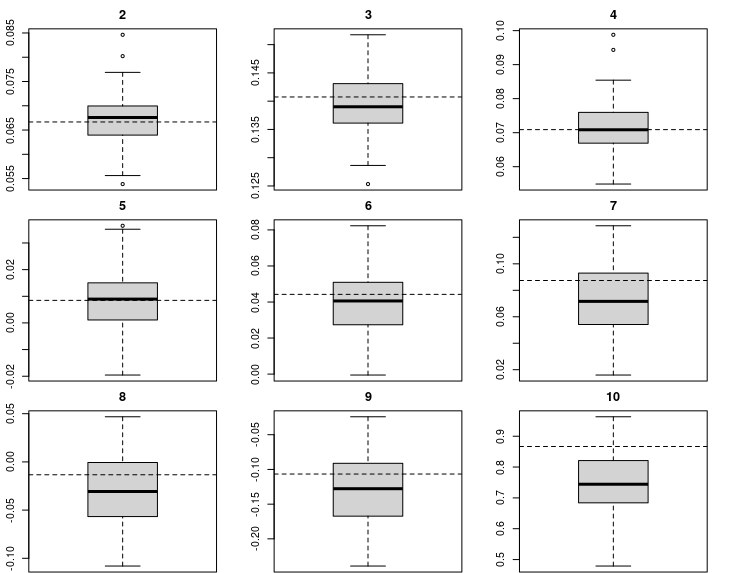}
	\caption{Boxplot of estimates for the simulation study}
	\label{fig:boxplot-zoom}
\end{figure}

\section{Acknowledgement}

This work was partially supported by the Natural Sciences and Engineering Research Council of Canada (Blier-Wong: 559169, Cossette: 04273; Marceau: 05605). The first author is also supported by grants by the Chaire en actuariat de l'Université Laval and the Quantact Actuarial and Financial Mathematics Laboratory.

\appendix

\section{Proof of parameters for exchangeable Beta}\label{app:proof-beta-exch}

We require the following Lemma, often used to prove Legendre's Duplication Formula. 
\begin{lemma}\label{lemma:beta-function}
	The following integral representation of the Beta function holds: 
	$$B(a, b) = 2\int_{0}^{1} x^{2a - 1}(1-x^2)^{b-1} \diff x.$$
\end{lemma}
\begin{proof}
	Using the definition of the Beta function, $B(a, b) = \int_{0}^{1}u^{a-1}(1-u)^{b-1} \diff u$, and substituting $u = x^2$ yields the desired result.
\end{proof}

We now prove the formulas in Example \ref{exa:definetti-beta}. From \eqref{eq:exch-esp-theta}, we obtain
\begin{equation}\label{eq:theta-beta}
	\theta_{k} =(-2)^kE_{\Lambda}\left\{\left(\Lambda - \frac{1}{2}\right)^k \right\}= (-2)^k \int_0^1 \frac{\Gamma(\alpha + \alpha)}{\Gamma(\alpha)\Gamma(\alpha)}\lambda^{\alpha - 1} (1 - \lambda)^{\alpha - 1}\left(\lambda - \frac{1}{2}\right)^k d\lambda.
\end{equation}
Using the substitution $\lambda = \frac{1 + v}{2}$, it follows that 
\begin{align}
	\theta_{k} 
	%&= (-2)^k \frac{\Gamma(2\alpha)}{\Gamma(\alpha)^2}\int_{-1}^1 \left(\frac{1 + v}{2}\right)^{\alpha - 1} \left(1 - \frac{1 + v}{2}\right)^{\alpha - 1}\left(\frac{1 + v}{2} - \frac{1}{2}\right)^k \frac{\diff v}{2}\nonumber\\
	% 	&= (-2)^k \frac{\Gamma(2\alpha)}{\Gamma(\alpha)^2}2^{-2\alpha - k}\int_{-1}^1 2\left(1+v\right)^{\alpha - 1} \left(1-v\right)^{\alpha - 1}v^k \diff v \nonumber\\
	&= (-1)^k \frac{\Gamma(2\alpha)}{\Gamma(\alpha)^2}4^{-\alpha}\int_{-1}^1 2\left(1 - v^2\right)^{\alpha - 1} v^k \diff v.\label{eq:theta-beta2}
\end{align}
Let us now solve the integral in \eqref{eq:theta-beta2}. One notices that $2\left(1 - v^2\right)^{\alpha - 1} v^k$ is an even function for $k \in \{2, 4, 6, \dots\}$ and an odd function for $k \in \{1, 3, 5, \dots\}$, so the integral equals 
\begin{equation}\label{eq:int-beta-exch}
	\int_{-1}^1 2\left(1 - v^2\right)^{\alpha - 1} v^k dv = \begin{cases}
		2\times \int_{0}^1 2\left(1 - v^2\right)^{\alpha - 1} v^k \diff v, & k \in \{2, 4, 6, \dots\}\\
		0, & k \in \{1, 3, 5, \dots\}
	\end{cases}.
\end{equation}
Therefore, we have $\theta_{k} = 0$ for $k = 1, 3, 5, \dots$. When $k$ is even, applying Lemma \ref{lemma:beta-function} on \eqref{eq:int-beta-exch} with $a= \frac{k + 1}{2}$ and $b = \alpha$ and simplifying, we obtain
% $$2\int_{0}^1 2\left(1 - v^2\right)^{\alpha - 1} v^k \diff v = 2\frac{\Gamma\left(\frac{k+1}{2}\right)\Gamma(\alpha)}{\Gamma\left(\alpha + \frac{k + 1}{2}\right)}.$$
% Therefore, for $k = 2, 4, 6, \dots$
\begin{align*}
	\theta_{k} &= 2\times 4^{-\alpha}\frac{\Gamma(2\alpha)}{\Gamma(\alpha)^2}\frac{\Gamma\left(\frac{k + 1}{2}\right)\Gamma(\alpha)}{\Gamma\left(\alpha + \frac{k + 1}{2}\right)}= 
	% 	2\times 4^{-\alpha}\frac{\Gamma(2\alpha)}{\Gamma(\alpha)}\frac{\Gamma\left(\frac{k + 1}{2}\right)}{\Gamma\left(\alpha + \frac{k + 1}{2}\right)} = 
	2\times 2^{-2\alpha}\frac{2^{2\alpha - 1} \Gamma\left(\alpha + \frac{1}{2}\right)}{\sqrt{\pi}}\frac{\Gamma\left(\frac{k + 1}{2}\right)}{\Gamma\left(\alpha + \frac{k + 1}{2}\right)},\label{eq:theta-beta3}
\end{align*}
the final equality follows by using Legendre's duplication formula (see, for example, \cite{abramowitz1964handbook}).
% , or applying Lemma \eqref{lemma:beta-function} with $a = \frac{1}{2}$ and $b = \alpha$, and obtains
% \begin{equation}
	% 	\theta_{k} = 2\times 2^{-2\alpha}\frac{2^{2\alpha - 1} \Gamma\left(\alpha + \frac{1}{2}\right)}{\sqrt{\pi}}\frac{\Gamma\left(\frac{k + 1}{2}\right)}{\Gamma\left(\alpha + \frac{k + 1}{2}\right)} = \frac{\Gamma\left(\alpha + \frac{1}{2}\right)\Gamma\left(\frac{k + 1}{2}\right)}{\sqrt{\pi}\Gamma\left(\alpha + \frac{k + 1}{2}\right)}.
	% \end{equation}

\section{Proof of supermodular lower bound}\label{app:lower-bound}

\subsection{A Lemma}
In this appendix, we identify the copula parameters corresponding to the lower bound for FGM copulas under the supermodular order from Theorem \ref{thm:lower-bound-i}. The following result will be useful. 
\begin{lemma}\label{lemma:eq-thetas}
	We have
	$$\theta_k = (-2)^kE\left\{\prod_{j = 1}^k \left(I_j - \frac{1}{2}\right)\right\} = E\left\{(-1)^{I_1 + \dots + I_k}\right\}.$$
\end{lemma}
\begin{proof}
	Since $I$ takes values 0 or 1, one replaces $1 - 2I = (-1)^I$ and simplifies.
	% , we have
	% $$(-2)^kE\left\{\prod_{j = 1}^k \left(I_j - \frac{1}{2}\right)\right\} = E_{I_1, \dots, I_k}\left\{\prod_{j = 1}^k \left(1 - 2I_j\right)\right\} = E\left\{\prod_{j = 1}^k \left(-1\right)^{I_j}\right\}.$$
\end{proof}

\subsection{Dependence parameters for even dimensions}
% For $d$ even, we have $r = d/2$ in \eqref{eq:end-fgm}, so the joint pmf of $(I_1, \dots, I_d)$ corresponding to the supermodular lower bound of eFGM copulas in Theorem \ref{thm:lower-bound-i} becomes
% $$\pr(I_1^- = i_1, \dots, I_d^- = i_d) = \begin{cases}
	% 	\binom{d}{d/2}^{-1}, & i_\bullet = d/2\\
	% 	0, & \text{otherwise}
	% \end{cases}.$$
Let $N_d^- = I_1^- + \dots + I_d^-$, then we have $\pr(N_d^- = d/2) = 1$. Consider the vector containing the first $k$ elements of $(I_1^-, \dots, I_d^-)$, noted $(I_{1\mid d} -, \dots, I_{k\mid d}^-)$ and the rv $N_{k\mid d}^- = I_{1\mid d}^- + \dots + I_{k\mid d}^-$. From Lemma \ref{lemma:eq-thetas}, we have
$$\theta_k = E\left\{(-1)^{I_{1\mid d}^- + \dots + I_{k\mid d}^-}\right\} = E\left\{(-1)^{N_{k\mid d}^-}\right\}.$$
One can interpret the pmf of $N_{k\mid d}^-$ as the probability of selecting without replacement $j$ ones from $k$ samples, from an urn containing $d/2$ ones and $d/2$ zeroes. Then, 
\begin{equation}\label{eq:pmf-nk}
	\pr(N_{k\mid d}^- = j) = \binom{d/2}{j}\binom{d/2}{k-j}/\binom{d}{k}, \quad j \in \{0, \dots, k\},
\end{equation}
which is the pmf of a hypergeometric distribution. From \cite[Section 6.3]{johnson2005univariate}, the (descending) factorial moment generating function for a hypergeometric distribution $X$ of $a$ successes, $b$ failures and $n$ picks is
% \begin{equation}\label{eq:fmgf}
	$E\left\{(1 + t)^X\right\} = {}_2F_1(-a, -n; -a-b; -t).$
	% \end{equation}
Replacing $t = -2$ yields the desired result. The second equality in \eqref{eq:end-param} follows from the identities ${}_2F_1(a, b; c; z) = {}_2F_1(b, a; c; z)$ and
$${}_2F_1(-n, b; 2b; 2) = \frac{n!2^{-n-1}\{1 + (-1)^n\}\Gamma(b + 1/2)}{(n/2)! \Gamma((n+1)/2 + b)}, \quad n \in \mathbb{N}.$$

% https://functions.wolfram.com/HypergeometricFunctions/Hypergeometric2F1/03/05/08/0002/
% https://functions.wolfram.com/07.23.03.0064.01
% http://functions.wolfram.com/07.23.03.0059.01

\subsection{Dependence parameters for odd dimensions}
% For $d$ odd, we have $r = (d-1)/2$ so the joint pmf of $(I_1, \dots, I_d)$ corresponding to the supermodular lower bound of eFGM copulas in Theorem \ref{thm:lower-bound-i} becomes
% \begin{equation}\label{eq:pmf-min-odd}
	% 	    \pr(I_1^- = i_1, \dots, I_d^- = i_d) = 
	% 	\begin{cases}
		% 		\frac{1}{2} \binom{d}{\frac{d-1}{2}}^{-1}, & i_\bullet = \frac{d-1}{2}\\
		% 		\frac{1}{2} \binom{d}{\frac{d+1}{2}}^{-1}, & i_\bullet = \frac{d+1}{2}\\
		% 		0, & \text{otherwise}
		% 	\end{cases}
	% \end{equation}
For $d$ odd, we have $\pr\{N^-_d = (d-1)/2\} = \pr\{N^-_d = (d+1)/2\} = 1/2$. By symmetry of Pascal's triangle, both binomial coefficients of \eqref{eq:end-fgm} are equal, so the pmf is equal over all cases where $N_d^-$ equals $(d-1)/2$ or $(d+1)/2$. Therefore, for $d$ odd, we have
\begin{equation}
	\pr(N_{k:d}^- = j) = \frac{1}{2}\binom{\frac{d-1}{2}}{j}\binom{d - \frac{d-1}{2}}{k-j}/\binom{d}{k} + \frac{1}{2}\binom{\frac{d+1}{2}}{j}\binom{d - \frac{d+1}{2}}{k-j}/\binom{d}{k}, \quad j \in \{0, \dots, k\},
\end{equation}
which is the average of the pmf of hypergeometric distributions with $(d+1)/2$ ones, $(d-1)/2$ zeroes and $k$ picks, and $(d-1)/2$ ones, $(d+1)/2$ zeroes and $k$ picks. Both cases are symmetric (with ones and zeroes swapped), so one can define the rv $N_{k:d}^-{}'$ which follows a hypergeometric distribution with $(d+1)/2$ ones, $(d-1)/2$ zeroes and $k$ picks, then similarly to the even case, we have
% $$\pr(N_{k:d}^- = j) = \frac{1}{2}\pr(N_{k:d}^-{}' = j) + \frac{1}{2}\pr(N_{k:d}^-{}' = k - j).$$
% With $(-1)^{-n} = (-1)^n$ for $n \in \mathbb{N}$, it follows that
$$E\left\{(-1)^{N_{k:d}^-{}}\right\} = \frac{1}{2}E\left\{(-1)^{N_{k:d}^-{}'}\right\} + \frac{1}{2}E\left\{(-1)^{k - N_{k:d}^-{}'}\right\} = \left(\frac{1}{2} + \frac{1}{2}(-1)^k\right) E\left\{(-1)^{N_{k:d}^-{}'}\right\}.$$
Applying the factorial moment generating function once again yields the desired result. 

\bibliographystyle{apalike}
\bibliography{paper-ref}

\end{document}